\documentclass[10pt]{article}

\usepackage[labelfont=bf]{caption}
\usepackage{authblk}
\usepackage{amsthm}
\usepackage{subcaption}
\usepackage{amsmath,amssymb}
\usepackage{tikz-cd}
\usepackage{xcolor}
\usepackage{url}
\usepackage{faktor}
\usepackage{enumerate}
\usepackage{geometry}
\usepackage{pgfplots}
\usepackage{mathtools,bbm}
\pgfplotsset{compat=1.18} 

\definecolor{blueXXredX}{rgb}{0.66, 0.33, 1.0}
\definecolor{blueXredXXXX}{rgb}{1.0, 0.22, 1.0}
\definecolor{darkBlue}{rgb}{0.16, 0.32, 0.75}
\definecolor{myGreen}{rgb}{0.18039216 0.49803922 0.09411765}
\definecolor{darkred}{rgb}{0.65, 0.0, 0.0}
\definecolor{CeruleanBlue}{rgb}{0.16, 0.32, 0.75}

\usepackage[colorlinks,linkcolor=myGreen,citecolor=myGreen,urlcolor=darkred]{hyperref}
\usepackage{cleveref}

\allowdisplaybreaks

\newtheorem{theorem}{Theorem}[section]
\theoremstyle{plain}

\newtheorem{corollary}[theorem]{Corollary}

\newtheorem{example}[theorem]{Example}

\newtheorem{lemma}[theorem]{Lemma}
\newtheorem{definition}[theorem]{Definition}

\newtheorem{proposition}[theorem]{Proposition}
\newtheorem{remark}[theorem]{Remark}

\usepackage[utf8]{inputenc}

\DeclareUnicodeCharacter{2080}{\textsubscript{0}}
\DeclareUnicodeCharacter{2081}{\textsubscript{1}}
\DeclareUnicodeCharacter{2082}{\textsubscript{2}}
\DeclareUnicodeCharacter{2083}{\textsubscript{3}}
\DeclareUnicodeCharacter{2084}{\textsubscript{4}}
\DeclareUnicodeCharacter{2085}{\textsubscript{5}}
\DeclareUnicodeCharacter{2086}{\textsubscript{6}}
\DeclareUnicodeCharacter{2087}{\textsubscript{7}}
\DeclareUnicodeCharacter{2088}{\textsubscript{8}}
\DeclareUnicodeCharacter{2089}{\textsubscript{9}}
\DeclareUnicodeCharacter{207D}{\textsuperscript{(}}
\DeclareUnicodeCharacter{207E}{\textsuperscript{)}}
\DeclareUnicodeCharacter{2295}{$\boldsymbol{\oplus}$}

\newcommand{\R}{\mathbb{R}}

\newcommand{\N}{\mathbb{N}}
\newcommand{\Z}{\mathbb{Z}}

\DeclareMathOperator{\Span}{span}

\DeclareMathOperator{\bary}{bary}

\newcommand\LieAlg{\mathfrak{g}}
\newcommand\G{\mathcal{G}}
\newcommand\ec{1}

\newcommand\axis{\mathsf{Axis}}
\newcommand\ax{\mathsf{Ax}}
\newcommand\rank{\mathsf{rk}}

\DeclareMathOperator{\eval}{eval}
\DeclareMathOperator{\level}{level}

\title{Learning barycenters from signature matrices}

\author[1]{Carlos Am\'endola\thanks{Email: \url{amendola@math.tu-berlin.de}}}
\author[1]{Leonard Schmitz\thanks{Email: \url{lschmitz@math.tu-berlin.de}}}

\affil[1]{\small
    {Institute of Mathematics, Technical University Berlin, Germany}
}

\date{}

\begin{document}

\maketitle

\begin{abstract}
The expected signature of a family of paths need not be a signature of a path itself. Motivated by this, we consider the notion of a Lie group barycenter introduced by Buser and Karcher to propose a barycenter on path signatures. We show that every element of the free nilpotent Lie group is a barycenter of a group sample, where all but one sample element can be fixed arbitrarily. In the case of piecewise linear paths, we study the problem of recovering an underlying path corresponding to the barycenter of signatures. We determine the minimal number of segments required to learn from signature matrices, providing explicit transformations to the associated congruence normal forms.
\end{abstract}

\\


\section{Introduction}

\subsubsection*{Path Signatures}
Chen's iterated-integral signature first appeared in the 50s, was later used by Lyons in the development of \emph{rough path theory}, and ultimately inspired Hairer's work on regularity structures. 
In the last decade, it has been applied successfully in differential equations
\cite{salvi2021signature}, finance \cite{cuchiero2023signature}, probability \cite{bayer2023optimal}, machine learning tasks \cite{chevyrev2025primersignaturemethodmachine}, and many other areas. The signature of a path is a collection of tensors that characterizes the path up to a mild equivalence relation. At each truncation level, the signature lies in the free nilpotent Lie group. For example, at the second truncation level, the \emph{signature matrix} stores the \emph{signed areas} of a path.

\subsubsection*{Lie group barycenters}
The notion of a \emph{Lie group barycenter} goes back to Buser's and Karcher's original work \cite{bib:BK1981}. In several Lie groups, this barycenter is a \emph{geometric group mean}; see \cite{bib:PL2020} for a recent exposition. This covers for instance the mean in the group of rotations \cite{moakher2002means} or the recent \emph{barycenter of iterated-integrals signatures} \cite{clausel2024barycenterfreenilpotentlie} from rough path theory.  

\subsubsection*{Learning paths from signatures}
 A natural task is to \emph{learn} the preimage of the signature when considered as a map -- an issue addressed in \cite{bib:PSS2019} for third-level tensors and piecewise linear paths using tools from computational algebraic geometry. It is well known that the \emph{expected signature} \cite{bib:Ni2012} of a stochastic process (or random path) does not lie in the Lie group, and therefore no underlying path can be recovered from it. In contrast, the Lie group barycenter provides a different notion of expectation that remains within the group and makes it possible to associate a path.
 
\subsubsection*{Contributions}
Little is known about the geometry of signature barycenters. For this reason, we study path recovery of Lie group barycenters for signatures of piecewise linear paths. 
We show that the barycenter is an epimorphism of algebraic varieties that is represented by a non-commutative polynomial, \Cref{thm:bary_is_algebraic}. 
This produces novel closed-form expressions of barycenters that depend only on the truncation level, e.g., see  \Cref{closed_formula_bary_k2} at matrix truncation. We apply the theory of congruence orbits to compute the minimal number of segments needed for our barycenter recovery task at the level of matrices, \Cref{thm:main_matrix}. In order to support future investigation on higher tensor levels, we provide well-written code in the computer algebra system \texttt{OSCAR} \cite{OSCAR,OSCAR-book}, available at
\url{https://github.com/leonardSchmitz/barycenter-signature}.

\subsubsection*{Outline}

 \Cref{sec:FreeLieGroups,sec:nonCommutativePolys} introduce relevant aspects on free nilpotent Lie groups and theory on non-commutative symbolic computation. 
 \Cref{sec:barycenter} is devoted to barycenters. 
\Cref{sec:signatureTensors,sec:path_recovery} treat signature tensors and path recovery, respectively. 
   \Cref{sec:matrix_case} unites all previous sections in the case of signature matrices. \Cref{sec:computational_aspects_oscar} reports on the computational aspects of this work.

\section{Free nilpotent Lie groups}\label{sec:FreeLieGroups}
To introduce the free nilpotent Lie group, we fix a \emph{dimension} $d$, and a \emph{truncation level} $k$, and consider the \emph{truncated tensor algebra} 
\begin{equation}\label{eq:truncTensSeq}
T_{d,k}
:=
\bigoplus_{\ell=0}^k(\R^d)^{\otimes\ell}
=  
\R\boldsymbol{\oplus}\R^d\boldsymbol{\oplus}\R^{d\times d}\boldsymbol{\oplus} \dots\boldsymbol{\oplus}\R^{d\times \dots\times d}
\end{equation}
serving as our $(d^{k+1}-1)/(d-1)$-dimensional ambient space. This space is a direct sum of $\ell$-tensors on $\R^d$. It is isomorphic to the free associative algebra over $d$ \emph{letters} modulo the two-sided ideal spanned by all \emph{words} with \emph{length} longer than $k$, e.g., see \cite[Section 2.1]{clausel2024barycenterfreenilpotentlie}. Throughout, we use the notation $[d]:=\{1,\dots,d\}$. Our representation $T_{d,k}$ stores all $\R$-basis coefficients in a sequence of $k$-tensors. After bilinearity, the compatible multiplication of $\ell$-tensors and $n$-tensors with $\ell+n\leq k$ translates to 
\begin{align}\label{eq:tensorMult}(\R^d)^{\otimes \ell}\otimes(\R^d)^{\otimes n}&\rightarrow(\R^d)^{\otimes (\ell+n)},\;
    \mathbf{a}\otimes\mathbf{b}\mapsto\mathbf{c}
    \end{align}
    where $\mathbf{c}_{i_1,\dots,i_{\ell+n}}:=\mathbf{a}_{i_1,\dots,i_\ell}\cdot\mathbf{b}_{i_{\ell+1},\dots,i_{\ell+n}}$ for all $i\in[d]^{\ell+n}$. Whenever $\ell+n> k$, we set the product to $0$. This product equips $T_{d,k}$ with an $\R$-algebra structure. For tensors (and sequences of tensors) we use bold letters, that is, $
\mathbf{y}=\mathbf{y}^{(0)}\boldsymbol{\oplus}\dots\boldsymbol{\oplus}\mathbf{y}^{(k)}\in T_{d,k}$ with projection to $\ell$-tensor $\mathbf{y}^{(\ell)}\in(\R^d)^{\otimes\ell}$ for all $0\leq\ell\leq k$. As in every associative algebra, we obtain a \emph{Lie bracket} via the \emph{commutator} $[\mathbf{y},\mathbf{z}]:=\mathbf{y}\otimes\mathbf{z} -\mathbf{z}\otimes \mathbf{y}$ for $\mathbf{y},\mathbf{z}\in T_{d,k}$. 
\par 
The \emph{free nilpotent Lie algebra}  
$\mathfrak{g}_{d,k}$ is the smallest Lie subalgebra of $T_{d,k}$ that contains
the standard basis $\mathrm{e}_{1},\dots,\mathrm{e}_{d}\in\R^{d}$. 
When considered as an $\R$-vector space, its dimension is 
\begin{equation}\label{eq:dim}
\lambda_{d,k}:=
\dim(\LieAlg_{d,k})=\sum_{\ell=1}^ k\frac1\ell\sum_{a\mid \ell}\mu(a)d^{\frac{\ell}{a}},\end{equation}
where $\mu$ denotes the \emph{M\"obius function}. 
 In this work we consider $\mathfrak{g}_{d,k}$ as an affine algebraic variety according to \cite[Theorem 3.1]{bib:R1993}, see also \cite[Lemma 4.1]{bib:AFS2019}. The relations are characterized by the vanishings of  \emph{shuffle} linear forms. 

\begin{example}
    If $d=k=2$ then $[\mathrm{e}_1,\mathrm{e}_2]
   =\begin{bsmallmatrix}
        0 &1\\-1 &0
    \end{bsmallmatrix}$ so $\mathfrak{g}_{d,k}$ is a subspace of $T_{2,2}\cong\R^7$ with dimension $\lambda_{2,2}=3$. Let $s_\varepsilon,s_1,s_2,s_{11},s_{12},s_{21}$ and $s_{22}$ be the $7$ variables that generate the associated coordinate ring. Then 
    $$\mathfrak{g}_{d,k}
    =\Span(\mathrm{e}_1,\mathrm{e}_2,\begin{bsmallmatrix}
        0 &1\\-1 &0
    \end{bsmallmatrix})
    =\mathcal{V}(s_\varepsilon,s_{11},s_{22},s_{12}+s_{21})$$
    where $s_\varepsilon$ ensures that the constant component is zero. The remaining relations are according to the possible $3$ distinct shuffles of $2$ letters. 
\end{example}
Since we work in the free nilpotent Lie algebra, the \emph{exponential}
\begin{equation}\label{eq:def_exp}
  \exp: \mathfrak{g}_{d,k} \to T_{d,k} ,\;
          \mathbf{z}          \mapsto \sum_{\ell=0}^\infty \frac1{\ell!}\mathbf{z}^{\otimes \ell}
\end{equation}
is a non-commutative polynomial map, that is, the summation index $\ell$ in \eqref{eq:def_exp} in fact ranges from $0$ to the truncation level $k$. We postpone the precise formalization of this statement to \Cref{lem:exp_log_inv_is_polynomial}. 
\par 
The \emph{free nilpotent Lie group} is the image of the exponential, 
$$\mathcal{G}_{d,k}:=\exp(\mathfrak{g}_{d,k}).$$
As the name suggests, $\mathcal{G}_{d,k}$ is a Lie group realized by \eqref{eq:def_exp}. Its product is inherited by $T_{d,k}$. We consider it as an affine algebraic variety characterized by the \emph{shuffle relations}, see \cite[Theorem 3.2]{bib:R1993} or \cite[Lemma 4.2]{bib:AFS2019}.
\begin{example}\label{ex:HeisenbergGroup}
  If $d=k=2$ then 
  $$\mathcal{G}_{2,2}=\mathcal{V}(s_\varepsilon-1,s_{1}^2-2s_{11},s_{1}s_{2}-s_{12}-s_{21},s_{2}^2-2s_{22})$$ is generated by $3$ shuffle relations as well as  $s_\varepsilon-1$ to ensure that all constant components are $1$. Considered as a group, $\mathcal{G}_{2,2}$ is isomorphic to the \emph{Heisenberg group}, see \cite[Exercise~9.17]{Friz_Victoir_2010}.
\end{example}
For the whole article, we restrict the exponential to its image,  $\exp: \mathfrak{g}_{d,k} \to {\mathcal{G}}_{d,k}$. 
Then, it is invertible with the inverse given by the \emph{logarithm},
\begin{align}\label{eq:def_log}
  \log: {\mathcal{G}}_{d,k} &\to \mathfrak{g}_{d,k},\;
    \mathbf{s} \mapsto \sum_{\ell=1}^\infty \frac{(-1)^{\ell+1}}\ell  {(\mathbf{s} - \ec)}^{\otimes \ell}.
\end{align}
In the following section, we see that both the exponential \eqref{eq:def_exp} and the logarithm \eqref{eq:def_log} are isomorphisms of affine algebraic varieties, that is $\mathfrak{g}_{d,k}\cong\mathcal{G}_{d,k}$.

\section{Non-commutative polynomials}\label{sec:nonCommutativePolys}

For any truncation level $k$ let     \begin{equation}\label{eq:univariate_free_algeba_Rs}\R\langle s\rangle:=\R\langle s^{(1)},\dots,s^{(k)}\rangle\end{equation} denote the free associative algebra over the $k$ symbols $s^{(i)}$, encoding the different levels $\mathbf{z}^{(i)}$ of a tensor sequence $\mathbf{z}=\mathbf{z}^{(0)}\boldsymbol{\oplus}\dots\boldsymbol{\oplus}\mathbf{z}^{(k)}\in T_{d,k}$.  

    In this algebra, we use the \emph{level grading}
    defined by the universal property of monoids and $\level(s^{(i)}):=i$ for $1\leq i\leq k$. 
    The evaluation of a (non-commutative) polynomial $f\in\R\langle s\rangle$ is a morphism of algebras, 
    \begin{align}\label{eq:evalNonComPolyTensors}\eval_{f}: T_{d,k}&\rightarrow T_{d,k} 
    \end{align}
defined via the universal property and $\eval_{s^{(i)}}(\mathbf{z}):=\mathbf{z}^{(i)}$, i.e., we replace all $s^{(i)}$ in $f$ by $\mathbf{z}^{(i)}$ and obtain a polynomial term in $T_{d,k}$. We avoid the notation $f(\mathbf{z})$ and instead write $\eval_f(\mathbf{z})$ to distinguish from the classical evaluation of commutative polynomials.  
\begin{lemma}\label{lem:non-com-poly-implies-poly-map}
    For every $f\in\R\langle s\rangle$, its evaluation $\eval_f$ is a morphism of varieties.
\end{lemma}
\begin{proof}It suffices to construct $(d^{k+1}-1)/(d-1)$ coordinate functions that define a (commutative) polynomial map $g$ on \eqref{eq:truncTensSeq} such that $g=\eval_f$. For variables $f=s^{(i)}$ we can choose those to be either identities or zeros. Addition in $\R\langle s\rangle$ translates to addition of coordinate functions and with \eqref{eq:tensorMult} we obtain coordinate functions for products in $\R\langle s\rangle$. 
\end{proof}
\begin{example}\label{ex:2nd_graded_comp_of_exp}
Let $d=k=2$ and $f=s^{(1)}+s^{(2)}+\frac12s^{(1)}s^{(1)}\in\R\langle s\rangle$, then  \begin{align*}\eval_f(\mathbf{z})
&=
\mathbf{z}^{(1)}
\boldsymbol{\oplus}
\left(\mathbf{z}^{(2)}+\frac12{\mathbf{z}^{(1)}}\otimes{\mathbf{z}^{(1)}}\right)=0 \boldsymbol{\oplus}\begin{bmatrix}
            \mathbf{z}_{{1}}^{(1)}\\
            \mathbf{z}_{{2}}^{(1)}
        \end{bmatrix}
        \boldsymbol{\oplus}
        \begin{bmatrix}
            \mathbf{z}_{{1,1}}^{(2)}+\frac12\mathbf{z}_{{1}}^{(1)}\mathbf{z}_{{1}}^{(1)} &
            \mathbf{z}_{{1,2}}^{(2)}+\frac12\mathbf{z}_{{1}}^{(1)}\mathbf{z}_{{2}}^{(1)}\\
            \mathbf{z}_{{2,1}}^{(2)}+\frac12\mathbf{z}_{{2}}^{(1)}\mathbf{z}_{{1}}^{(1)} 
            &
            \mathbf{z}_{{2,2}}^{(2)}+\frac12\mathbf{z}_{{2}}^{(1)}\mathbf{z}_{{2}}^{(1)}
        \end{bmatrix}
    \end{align*}
    is evaluated at $\mathbf{z}=\mathbf{z}^{(0)}\boldsymbol{\oplus}\mathbf{z}^{(1)}\boldsymbol{\oplus}\mathbf{z}^{(2)}\in T_{2,2}$, from which its $7$ coordinate functions can be read. 
\end{example}
We say that a (non-commutative) polynomial map is \emph{represented} by $f$ if it is equal to $\eval_f$.
\begin{lemma}\label{lem:expLogMulLinIn_ck}
Let $k\in \N$ and $R=\R\langle s^{(1)},\dots,s^{(k-1)}\rangle$. The exponential $\exp$ and the logarithm $\log$ are represented by elements of $s^{(k)}+R$, while group inversion  $(\cdot)^{-1}:\mathcal{G}_{d,k}\rightarrow\mathcal{G}_{d,k}$ is represented by an element of $-s^{(k)}+R$.
\end{lemma}
\begin{proof}
    We use the level grading and non-commutative polynomials representing the exponential \eqref{eq:def_exp} and the logarithm  \eqref{eq:def_log}. The group inversion is represented by a non-commutative polynomial since every formal power series with a nonzero constant component is invertible.  
\end{proof}

\begin{example}
  In cubic truncation $k=3$, we have
    $$\mathbf{z}^{-1}
    =
    1 \boldsymbol{\oplus}
    \left(-\mathbf{z}^{(1)}\right)
     \boldsymbol{\oplus}
    \left(-\mathbf{z}^{(2)}+
    \mathbf{z}^{(1)}\otimes \mathbf{z}^{(1)}
    \right)
     \boldsymbol{\oplus}
    \left(-\mathbf{z}^{(3)}
    +\mathbf{z}^{(2)}\otimes \mathbf{z}^{(1)}
    +\mathbf{z}^{(1)}\otimes \mathbf{z}^{(2)}
    -\mathbf{z}^{(1)}\otimes \mathbf{z}^{(1)}\otimes \mathbf{z}^{(1)}
    \right)$$
    for every $\mathbf{z}\in\mathcal{G}_{d,3}$ so the group inverse is represented by a non-commutative polynomial of $8$ terms. 
\end{example}
\begin{corollary}\label{lem:exp_log_inv_is_polynomial}
The exponential, the logarithm and the group inversion in $\mathcal{G}_{d,k}$ are isomorphisms of varieties. 
\end{corollary}

This is an immediate consequence of \Cref{lem:non-com-poly-implies-poly-map} and \Cref{lem:expLogMulLinIn_ck}. 
Further on (see for instance \Cref{lem:prod_is_polynomial} and  \Cref{lem:baryIsNonComPolyMap}), we are handling $N$ tensor sequences at the same time. For this, we generalize \eqref{eq:univariate_free_algeba_Rs} to the \emph{multivariate setting}, 
\begin{equation}\label{eq:multivariate_free_algeba_Rs}\R\langle s_1,\dots,s_N\rangle:=\R\langle s^{(j)}_{i}\mid i\in[N],j\in[ k]\rangle,\end{equation}
with $\level(s_j^{(i)}):=i$ and an analogous \emph{evaluation} $\eval_q:T_{d,k}^N\rightarrow T_{d,k}$
for any $q\in\R\langle s_1,\dots,s_N\rangle$. 

\begin{lemma}\label{lem:MulLinIn_ck}
The group multiplication $\mathcal{G}_{d,k}\times\mathcal{G}_{d,k}\rightarrow\mathcal{G}_{d,k},(\mathbf{z}_1,\mathbf{z}_2)\mapsto {\mathbf{z}_1}\cdot{\mathbf{z}_2}$ is represented by an element of 
$s_1^{(k)}+s_2^{(k)}+\R\langle s_1^{(1)},\dots,s_1^{(k-1)},s_2^{(1)},\dots,s_2^{(k-1)}\rangle$. 
\end{lemma}

\begin{example}For all $\mathbf{z}_1,\mathbf{z}_2\in\mathcal{G}_{d,3}$ its product is represented by a polynomial of $10$ terms,
$$\mathbf{z}_1\cdot\mathbf{z}_2
=1
 \boldsymbol{\oplus}
\left(
\mathbf{z}^{(1)}_1+\mathbf{z}^{(1)}_2\right)
 \boldsymbol{\oplus}
\left(
\mathbf{z}^{(2)}_1+\mathbf{z}^{(2)}_2+\mathbf{z}^{(1)}_1\otimes\mathbf{z}^{(1)}_2
\right)
 \boldsymbol{\oplus}
\left(
\mathbf{z}^{(3)}_1+\mathbf{z}^{(3)}_2+\mathbf{z}^{(2)}_1\otimes\mathbf{z}^{(1)}_2
+\mathbf{z}^{(1)}_1\otimes\mathbf{z}^{(2)}_2
\right).
$$
\end{example}

\begin{corollary}\label{lem:prod_is_polynomial}
The group multiplication in $\mathcal{G}_{d,k}$ is an epimorphism of varieties. 
\end{corollary}

We recall the important notion of a \emph{congruence transform} on tensors, e.g., compare \cite[Section~3]{bib:PSS2019}.  
Let $\mathbf{c}\in(\R^{m})^{\otimes k}$
be a \emph{core tensor} and $A\in\R^{d\times m}$
a \emph{linear transform}. Then, the \emph{matrix-tensor congruence transform} $A\cdot\mathbf{c}\in(\R^{d})^{\otimes k}$ is a $k$-tensor with entries
\begin{equation}\label{eq:def_tensorMatrixCongruence}\left(A\cdot\mathbf{c}\right)_{w_1,\dots,w_k}:=\sum_{i_1=1}^m\dots\sum_{i_k=1}^mA_{w_1,i_1}\dots A_{w_k,i_k}\mathbf{c}_{i_1,\dots,i_k}\end{equation}
for every $ w_1,\dots,w_k\in[ d]$. Whenever $m$ is significantly smaller than $d$, the factorization $A\cdot\mathbf{c}$ is known as the \emph{Tucker format}. The matrix-tensor congruence naturally extends to sequences of tensors by action on its graded components with respect to the level, that is
$\R^{d\times m}\times T_{m,k}\rightarrow T_{d,k}$
with $(A\cdot\mathbf{z})^{(j)}:=A\cdot\mathbf{z}^{(j)}$ for every $\mathbf{z}\in T_{m,k}$ and $ j\in[k]$. 

\begin{lemma}\label{lem:evalEquivariant} 
The evaluation is equivariant under congruence transforms, i.e., $$A\cdot\eval_f(\mathbf{z})=\eval_f(A\cdot\mathbf{z})$$
for every $f\in\R\langle s_1,\dots,s_N\rangle$, $\mathbf{z}\in T_{m,k}^N$ and $A\in \R^{d\times m}$. 
\end{lemma}
\begin{proof}
    Congruence and evaluation are both linear, so it suffices to consider evaluations of monomials,  \begin{align*}
    A\cdot\eval_{s_{j_1}^{(i_1)}\dots s_{j_t}^{(i_t)}}(\mathbf{z})
    &=\left(A^{\otimes i_1}\otimes \dots \otimes A^{\otimes i_t}\right)\left(\mathbf{z}_{j_1}^{(i_1)}\otimes \dots \otimes \mathbf{z}_{j_t}^{(i_t)}\right)\\
    &=\left(A\cdot\mathbf{z}_{j_1}^{(i_1)}\right)\otimes\dots\otimes \left(A\cdot\mathbf{z}_{j_t}^{(i_t)}\right)
    =
\eval_{s_{j_1}^{(i_1)}\dots s_{j_t}^{(i_t)}}(A\cdot\mathbf{z})
    \end{align*}
for every $ i_1+\dots+i_t\leq k$ and $j_1,\dots,j_t\in[N]$. 
\end{proof}

\section{Barycenters}\label{sec:barycenter}
We recall the notion of a Lie group barycenter according to \cite[Definition~8.1.4]{bib:BK1981}. A more recent reference is \cite[Definition 11]{bib:PL2020}. 
\begin{definition}[Buser and Karcher]\label{def:groupmean}
  Let $\mathbf{x}_1,\dots,\mathbf{x}_N$ be elements of a Lie group with a globally defined logarithm. 
We say that a group element $\mathbf{m}$  is a \emph{barycenter}
of the \emph{sample} $\mathbf{x}=(\mathbf{x}_1,\dots,\mathbf{x}_N)$ if the following equation holds in the associated Lie algebra: 
  \begin{align}
  \label{eq:barycenterCondition}
    %0 = \sum_{i=1}^N \log( \mathbf{m}^{-1} \cdot \mathbf{x}_i ).
    \sum_{i=1}^N \log( \mathbf{m}^{-1} \cdot\mathbf{x}_i ) = 0.
  \end{align}
\end{definition}

\begin{theorem}[{\cite[Theorem 10]{clausel2024barycenterfreenilpotentlie}}]\label{thm:groupmean}
    In the free nilpotent Lie group $\mathcal{G}_{d,k}$ the barycenter exists and is unique. In other words, we have a map  
\begin{align*}\mathsf{bary}:\mathcal{G}_{d,k}^N&\rightarrow\mathcal{G}_{d,k}\\\mathbf{x}&\mapsto\mathbf{m}
\end{align*}
where $\mathbf{m}$ is the unique solution of \eqref{eq:barycenterCondition} with respect to our group sample  $\mathbf{x}\in\mathcal{G}_{d,k}^N$. 
\end{theorem}

Note that \cite{clausel2024barycenterfreenilpotentlie} presents \Cref{def:groupmean,thm:groupmean} for arbitrary probability distributions, thereby covering \emph{weighted} barycenters. We restrict ourselves to uniformly distributed weights for simplicity, and note that all results can be lifted to the more general setting of arbitrary distributions. Importantly, the barycenter is not simply given by the pointwise expectation of tensor entries, as is the case in the definition of the expected signature \cite[Theorem 2.3 and Remark 2.6.2]{bib:Ni2012}. Moreover, for truncation level $k \geq 3$, it is also not the pointwise expectation in the associated Lie algebra; see \cite[Example 4.9]{clausel2024barycenterfreenilpotentlie}.

We recall that the barycenter is left- and right-equivariant, and that it commutes with group inversion. 
\begin{theorem}\label{thm:equivariance_bary}
For all $d,k,N\in\N$ and  $\mathbf{x}_{1},\dots,\mathbf{x}_N,\mathbf{g}\in\G_{d,k}$ we have 
\begin{align*}
\bary(\mathbf{g}\cdot\mathbf{x}_1,\dots,\mathbf{g}\cdot\mathbf{x}_N)&=\mathbf{g}\cdot\bary(\mathbf{x}_1,\dots,\mathbf{x}_N),\\
\bary(\mathbf{x}_1\cdot\mathbf{g},\dots,\mathbf{x}_N\cdot\mathbf{g})&=\bary(\mathbf{x}_1,\dots,\mathbf{x}_N)\cdot\mathbf{g}\text{, and}\\
\bary(\mathbf{x}_1^{-1},\dots,\mathbf{x}_N^{-1})&=\bary(\mathbf{x}_1,\dots,\mathbf{x}_N)^{-1}.
\end{align*}
\end{theorem}

A proof can be found in \cite[Theorem 5.13]{bib:PL2020} or \cite[Theorem 13]{clausel2024barycenterfreenilpotentlie}. Note that these laws of compatibility appear also in the theory of \emph{proper scoring rules}, see \cite[Theorem 4.2]{BO24}. 
\par
In the following we show that for $N=2$ group sample elements, the local solution of \cite[Proposition~5.2]{bib:PL2020} is a \emph{global solution} in $\mathcal{G}_{d,k}$ for every $d$ and $k$. 

\begin{proposition}\label{prop:baryFormulaN2}
    If we have $N=2$ group samples, then $$\bary(\mathbf{x}_1,\mathbf{x}_2)=\mathbf{x}_1\cdot\exp \left(\frac12\log(\mathbf{x}_1^{-1}\cdot\mathbf{x}_2) \right)$$
    for all $d,k\in\N$ and $\mathbf{x}_1,\mathbf{x}_2\in\mathcal{G}_{d,k}$. 
\end{proposition}

\begin{proof}
    We first show \begin{equation}\bary(1,\mathbf{z})=\exp\left(\frac12\log(\mathbf{z})\right)\label{eq:baryWithOrigin1}
    \end{equation}
    for every $\mathbf{z}\in\mathcal{G}_{d,k}$. In fact $\exp(\frac12\log(\mathbf{z}))$ is a solution of \eqref{eq:barycenterCondition} with group sample elements $1$ and $\mathbf{z}$,
    \begin{align*}
        &\log\left(\exp\left(\frac12\log\left(\mathbf{z}\right)\right)^{-1}\cdot1\right)
        +
        \log\left(\exp\left(\frac12\log\left(\mathbf{z}\right)\right)^{-1}\cdot\mathbf{z}\right)\\
        &=-\frac12\log(\mathbf{z})
        +\log\left(\exp\left(-\frac12\log(\mathbf{z})\right)\cdot\exp(\log(\mathbf{z}))\right)\\
        &=-\frac12\log(\mathbf{z})
        +\frac12\log(\mathbf{z})=0,
    \end{align*}
    where we use the Baker–Campbell–Hausdorff formula with commuting arguments in the second summand; see for instance \cite[Theorem 7.24]{Friz_Victoir_2010}. Therefore, by the uniqueness of the barycenter from \Cref{thm:groupmean}, we conclude \eqref{eq:baryWithOrigin1}. With equivariance from \Cref{thm:equivariance_bary},
    \begin{align*}
    \bary(\mathbf{x}_1,\mathbf{x}_2)
&=\mathbf{x}_1\cdot\mathbf{x}_1^{-1}\cdot\bary(\mathbf{x}_1,\mathbf{x}_2)=\mathbf{x}_1\cdot\bary(1,\mathbf{x}_1^{-1}\cdot\mathbf{x}_2)
=\mathbf{x}_1\cdot\exp(\frac{1}2\log(\mathbf{x}_1^{-1}\cdot\mathbf{x}_2))
\end{align*}
using \eqref{eq:baryWithOrigin1} for $\mathbf{z}:=\mathbf{x}^{-1}_1\mathbf{x}_2$. 
\end{proof}

We provide an infinite class of group elements that have a barycenter solution that does not depend on the truncation level $k$ nor the dimension $d$. 
\begin{proposition}\label{prop:exponents_sample}
    For all $d,k,N\in\N$, $\mathbf{x}\in\mathcal{G}_{d,k}$ and $u\in\N^N$ with $u_1+\dots+u_N\in N\Z$ we have 
    $$\bary(\mathbf{x}^{u_1},\dots,\mathbf{x}^{u_N})=\mathbf{x}^{(u_1+\dots+u_N)/N}.$$
\end{proposition}
\begin{proof}
    It suffices to show that $\mathbf{m}:=\mathbf{x}^{(u_1+\dots+u_N)/N}$ is a solution of \eqref{eq:barycenterCondition}. Clearly $\mathbf{m}^{-1}$ and $\mathbf{x}^{u_i}$ commute for all $i\in[N]$, so $\log(\mathbf{m}^{-1}\mathbf{x}^{u_i})=-\log(\mathbf{m})+{u_i}\log(\mathbf{x})$, and hence 
    \begin{align*}
        \sum_{i=1}^N \log( \mathbf{m}^{-1} \cdot\mathbf{x}_i^{u_i} )=-N\log(\mathbf{m})+\sum_{i=1}^Nu_i\log(\mathbf{x})=0. \tag*{\qedhere}
    \end{align*} 
\end{proof}
\begin{example} We have 
$\bary(\mathbf{x}^2,1)=\mathbf{x}$, 
 $\bary(\mathbf{x},\mathbf{x}^{-1})=1$, and $\bary(\mathbf{x}^6,\mathbf{x}^2,\mathbf{x})=\mathbf{x}^3$.  
\end{example}

We state our first main result, showing that every element of the free nilpotent Lie group is a barycenter of a group sample, where all but one sample element can be fixed arbitrarily.

\begin{theorem}\label{thm:bary_is_algebraic}
    For every $d,k,N\in\N$, the barycenter  $\bary:\mathcal{G}_{d,k}^N\rightarrow\mathcal{G}_{d,k}$ is an epimorphism of algebraic varieties. If $N-1$ of the $N$ group sample elements are arbitrarily fixed, then it is an isomorphism. 
    \end{theorem}

We provide a proof in the remaining of this section. We start with the following elementary properties of the barycenter map. For this proof, we rely on the techniques developed in \cite{clausel2024barycenterfreenilpotentlie}. In particular, we refer to \cite[Section 2.3]{clausel2024barycenterfreenilpotentlie} for further details of Lyndon bases. 

\begin{proposition}\label{lem:bary_idempotent}
For every $d,k,N\in\N$, the barycenter map $\bary:\G_{d,k}^N\rightarrow\G_{d,k}$ is
\begin{enumerate}[(i)]
    \item\label{lem:bary_idempotent1} surjective,
    \item\label{lem:bary_idempotent2} invariant under permutation of the sample, and 
    \item\label{lem:bary_idempotent3} bijective, when $N-1$ of $N$ sample elements are fixed. 
\end{enumerate}
\end{proposition}
\begin{proof}
For part \eqref{lem:bary_idempotent1} it suffices to show $\mathbf{x}
    =\bary(\mathbf{x},\dots,\mathbf{x})$ for all $\mathbf{x}\in\G_{d,k}$. 
 With \eqref{eq:barycenterCondition} we obtain $0=\log(\mathbf{m}^{-1}\cdot\mathbf{x})$, hence $\ec=\mathbf{m}^{-1}\cdot\mathbf{x}$, and therefore $\mathbf{m}=\mathbf{x}$, so the claim follows from the uniqueness according to \Cref{thm:groupmean}. Part \eqref{lem:bary_idempotent2} is clear since the summands in \eqref{eq:barycenterCondition} commute.
For part \eqref{lem:bary_idempotent3} we show that for given $\mathbf{x}_{1},\dots,\mathbf{x}_{N-1},\mathbf{y}\in\G_{d,k}$ there is a unique $\mathbf{x}_N\in\G_{d,k}$ such that $\bary(\mathbf{x}_1,\dots,\mathbf{x}_N)=\mathbf{y}$. 
First, we assume that $\mathbf{x}_N$ exists and show that it is uniquely determined. 
Let $\mathcal{B}$ denote the Lyndon basis for $\mathfrak{g}_{d,k}$. Let $m\in\R^{\lambda_{d,k}}$ with $\mathbf{y}=\exp(\mathcal{B}m)$ and $x\in\R^{N\times \lambda_{d,k}}$ with $\mathbf{x}_i=\exp(\mathcal{B}x_{i,\bullet})$ for all $i\in [N]$ be given. 
We use the polynomials $g_j$ and $r_j$ according to \cite[Lemma 15]{clausel2024barycenterfreenilpotentlie}, so that  
    \begin{equation}\label{eq:commutativePolysDefBarys}m_j=g_j(x)=\frac1N\sum_{i=1}^Nx_{i,j}+r_j(x_{\bullet,j-1},\dots,x_{\bullet,1})\end{equation}
    for all $j\in[\lambda_{d,k}]$, and such that the diagram  
    $$\begin{tikzcd}[column sep = 3cm]
    \G_{d,k}^N 
    \arrow[d,"\bary"'] 
    \arrow[r,"\log\times\cdots\times\log"] 
    & 
    \LieAlg_{d,k}^N 
    \arrow[r,"\mathcal{B}^*\times\cdots\times\mathcal{B}^*"] 
    &
    \R^{N\times \lambda_{d,k}}
    \arrow[d,"g",dotted] 
    \\
    \G_{d,k} 
    & \LieAlg_{d,k}
    \arrow[l,"\exp"]
    &
    \R^{\lambda_{d,k}}
    \arrow[l,"\mathcal{B}"]
    \end{tikzcd}$$
commutes. Here, $\mathcal{B}^*$ denotes the dual of the Lyndon basis. Note that all horizontal arrows are isomorphisms. 
  Therefore with \eqref{eq:commutativePolysDefBarys}, 
    \begin{equation}\label{eq:invers:bary}x_{N,j}=Nm_j-\sum_{i=1}^{N-1}x_{i,j}-Nr_j(x_{\bullet,j-1},\dots,x_{\bullet,1})
    \end{equation}
    and we obtain the unique $\mathbf{x}_N\in\mathcal{G}_{d,k}$ with $\bary(\mathbf{x}_1,\dots,\mathbf{x}_N)=\mathbf{y}$. The explicit formula \eqref{eq:invers:bary} also yields the existence of $\mathbf{x}_N$.
\end{proof}

We now take a different approach, inspired by \cite[Algorithm 3 \& Proposition 34]{clausel2024barycenterfreenilpotentlie}, and use non-commutative polynomials and evaluations from \Cref{sec:nonCommutativePolys} to construct the barycenter map. Note that this approach yields new closed-form expressions, relying on previous results on the existence of the barycenter using Lie algebras,  \Cref{thm:groupmean}.  For this, let $\bary^{(j)}(\mathbf{x})$ denote the projection of $\bary(\mathbf{x})$ to the $j$-th tensor component.

\begin{proposition}
\label{lem:baryIsNonComPolyMap}
 For every $d,k,N\in\N$, the barycenter is a non-commutative polynomial map, that is, there exists a polynomial $q\in\R\langle s_1,\dots,s_N\rangle$ such that $\bary=\eval_q$.
\end{proposition}
\begin{proof}
We extend $\R\langle s_1,\dots,s_N\rangle$ by $k$ more variables $y^{(j)}$ for our barycenter of $N$ samples. 
Let $h_{\log}$, $h_{\mathsf{prod}}$ and $h_{\mathsf{inv}}$ be non-commutative polynomials that represent the logarithm, multiplication, and inversion according to \Cref{lem:expLogMulLinIn_ck,lem:MulLinIn_ck}, respectively. Then, 
\begin{equation}\label{eq:construction}g:=\frac1N\sum_{i=1}^Nh_{\mathsf{log}}\left(h_{\mathsf{prod}}\left(h_{\mathsf{inv}}\left(\sum_{j=1}^k{y^{(j)}}+1\right),\sum_{j=1}^k{s^{(j)}_i}+1\right)\right)\end{equation}
satisfies  
$\eval_g(\mathbf{x},\bary(\mathbf{x}))=0$
for every sample $\mathbf{x}\in\mathcal{G}^N$ due to \eqref{eq:barycenterCondition}. 
Let $\pi_j$ denote the projection to the graded component of level $j$.
Again with \Cref{lem:expLogMulLinIn_ck,lem:MulLinIn_ck}, 
\begin{equation}\label{eq:fj}
f_j:=\pi_j(g+y^{(j)})\in\R\langle s_1^{(1)},\dots,s_1^{(j)},s_2^{(1)},\dots,s_N^{(j)},y^{(1)},\dots,y^{(j-1)}\rangle
\end{equation}
which evaluates to 
$\eval_{f_j}(\mathbf{x},\bary(\mathbf{x}))=\bary^{(j)}(\mathbf{x})$
for all $j\in[k]$. 
By inserting $f_1,\dots,f_{j-1}$ into $f_j$ we obtain 
$$p_j:=f_j(s_1,\dots,s_N,f_{1},\dots,f_{j-1},0_{k-j+1})\in\R\langle s_1,\dots,s_N\rangle$$
which also evaluates to $\eval_{p_j}(\mathbf{x})=\bary^{(j)}(\mathbf{x})$.
 By summation of all polynomials $p_j$ representing the graded components of the barycenter, we obtain $
      q:=1+\sum_{j=1}^kp_j$
  with $\eval_q(\mathbf{x})=\bary(\mathbf{x})$. 
\end{proof}

We have now collected all the tools to prove our first main result. 
\begin{proof}[Proof of \Cref{thm:bary_is_algebraic}]Combining \Cref{lem:baryIsNonComPolyMap} and \Cref{lem:non-com-poly-implies-poly-map} we see that the barycenter is a morphism. In \Cref{lem:bary_idempotent} we show that it is surjective, and bijective when $N-1$ samples are fixed. 
We can construct a non-commutative polynomial inverse with the same techniques as in the proof of \Cref{lem:baryIsNonComPolyMap} by exchanging the roles of $y^{(j)}$ and $s^{(j)}$ in every truncation step $j$. This is similar to the construction \eqref{eq:construction} and can be found in the Github repository.
 \end{proof}

\begin{remark}\label{rem:polynomialViaBCH}
 The polynomial from \Cref{lem:baryIsNonComPolyMap} factorizes in the following sense:  there exists a polynomial $f\in\R\langle s_1,\dots,s_{N}\rangle$ such that the diagram
    $$\begin{tikzcd}[column sep = 3cm]
    \G_{d,k}^N 
    \arrow[d,"\bary"'] 
    \arrow[r,"\log\times\cdots\times\log"] 
    & 
    \LieAlg_{d,k}^N 
    \arrow[d,"\eval_f",dotted] 
    \\
    \G_{d,k} 
    & \LieAlg_{d,k}
    \arrow[l,"\exp"]
    \end{tikzcd}$$
commutes. 
    Note that \cite[Corollary 41]{clausel2024barycenterfreenilpotentlie} provides an alternative polynomial that can take the role of $f$, using the \emph{asymmetrized  Baker--Campbell--Hausdorff series}.  
\end{remark}

We close this section with a novel closed-form expression for barycenters based on tensor projections up to quadratic truncation.
In \Cref{ex:HeisenbergGroup} we recall that $\mathcal{G}_{2,2}$ is the Heisenberg group. A closed-form expression for barycenters in the Heisenberg group is provided in \cite[Proposition~3]{bib:PL2020}.   
The following expression coincides with this formula, but generalizes to arbitrary $d\geq 2$.

\begin{theorem}\label{closed_formula_bary_k2}If $k=2$, then every $\mathbf{x}\in\G_{d,2}^N$ satisfies 
\begin{align*}
\bary(\mathbf{x})=1
&\boldsymbol{\oplus}\left(\frac1N\sum_{i=1}^N\mathbf{x}_i^{(1)}\right)\\
&\boldsymbol{\oplus}\left(\frac1N\sum_{i=1}^N\mathbf{x}_i^{(2)}-\frac1{2N}\sum_{i=1}^N\mathbf{x}_i^{(1)}\otimes\mathbf{x}_i^{(1)}+\frac{1}{2N^2}\sum_{i_1=1}^N\sum_{i_2=1}^N\mathbf{x}_{i_1}^{(1)}\otimes\mathbf{x}_{i_2}^{(1)}\right).
   \end{align*}
\end{theorem}
   \begin{proof}
We use the asymmetrized Baker--Campbell--Hausdorff series \cite[Remark 40]{clausel2024barycenterfreenilpotentlie} and obtain  
\begin{equation}\label{eq:aBCHk2}f:=\frac{1}{N}\sum_{i=1}^N(s_i^{(1)}+s_i^{(2)})\end{equation}
for the factorization according to \Cref{rem:polynomialViaBCH}. 
Let $h_{\exp}$ and  $h_{\log}$ be  polynomials representing exponential and logarithm, respectively.  Hence we can choose the concatenated polynomial  
\begin{align*}
 q&:=h_{\exp}\left(f\left(h_{\log}\left(s_1^{(2)}+s_1^{(1)}+1\right),\dots,h_{\log}\left(s_N^{(2)}+s_N^{(1)}+1\right)\right)\right)
\\
&=h_{\exp}\left(f\left(s_1^{(2)}-\frac12s_1^{(1)}s_1^{(1)}+s_1^{(1)},\dots,s_N^{(2)}-\frac12s_N^{(1)}s_N^{(1)}+s_N^{(1)}\right)\right)\\
&=h_{\exp}\left(\frac1N\sum_{i=1}^N\left(s_i^{(2)}-\frac12s_i^{(1)}s_i^{(1)}+s_i^{(1)}\right)\right)
\\
&=
   \frac1N\sum_{i=1}^N\left(s_{i}^{(2)}+s_i^{(1)}\right)-\frac1{2N}\sum_{i=1}^Ns_i^{(1)}s_i^{(1)}+\frac{1}{2N^2}\sum_{i_1=1}^N\sum_{i_2=1}^Ns_{i_1}^{(1)}s_{i_2}^{(1)}+1   
\end{align*} 
as an alternative to the construction in \Cref{lem:baryIsNonComPolyMap}, and hence the claimed formula. 
\end{proof}

\section{Signature tensors}\label{sec:signatureTensors}
A \emph{path} 
$
X : [0,1] \to \R^d
$
is a (smooth enough) continuous curve such that the following \emph{iterated integrals} \eqref{eq:integrals} are defined. Without loss of generality, we assume that all paths start at the origin, i.e. $X(0)=0$. Since signatures are invariant under translation and since we study paths solely through their signatures, this is a common and natural assumption. 
\begin{definition}
 For a path $X$, the $k$-truncated \emph{iterated-integrals signature} $\sigma=\sigma(X)\in T_{d,k}$ is a direct sum $$\sigma:=\sigma^{(0)}\boldsymbol{\oplus}\sigma^{(1)}\boldsymbol{\oplus}\dots\boldsymbol{\oplus}\sigma^{(k)}$$ of \emph{signature tensors} $\sigma^{(\ell)}\in(\R^d)^{\otimes\ell}$ with entries,  
\begin{align}
  \label{eq:integrals}
  \sigma^{(\ell)}_{w_1\dots w_\ell}:=\int_0^1\int_0^{t_\ell}\dots\int_0^{t_2} \dot X_{w_1}(t_1) \dots \dot X_{w_\ell}(t_\ell)\,\mathrm dt_1 \dots \mathrm dt_\ell
\end{align}
corresponding to
$w_1,\dots ,w_\ell\in[d]$. 
Here, $\dot X$ denotes the differential of $X$. For the constant component, we set $\sigma^{(0)}:=1$. 
\end{definition}

\begin{theorem}[{\cite[Theorem 7.30]{Friz_Victoir_2010}}]\label{thm:sig_are_grouplike}
  For every truncation level $k$ and $d$-dimensional path $X$, its signature is a Lie group element, i.e. $\sigma(X)\in\G_{d,k}$.
\end{theorem}

For two paths $X,Y:[0,1]\rightarrow\R^d$, let $X\star Y$ denote its \emph{path concatenation}, that is, let 
\begin{equation}\label{eq:pathconc}
    (X\star Y):[0,1]\rightarrow \R^d,\;t\mapsto
\begin{cases}X(2t)&t\in[0,\frac12)\\
X(1)+Y(2t-1)&t\in[\frac12,1].
\end{cases}
\end{equation}
Then, the signature is multiplicative with respect to path concatenation and multiplication in the Lie group. This connection between the path space and the Lie group is called \emph{Chen's identity}. We refer to \cite[Theorem 7.11]{Friz_Victoir_2010} for a proof. 

\begin{theorem}[Chen]\label{thm:Chen}For all paths $X,Y:[0,1]\rightarrow\R^d$, 
$$\sigma(X\star Y)=\sigma(X)\cdot\sigma(Y).$$
\end{theorem}

From \cite [Equation (8)]{bib:PSS2019} we recall \emph{canonical axis paths}. These paths will be our \emph{dictionary} for path recovery in \Cref{sec:path_recovery,sec:matrix_case}. For a closed-form expression of signatures for axis paths, compare \cite[Example 2.1]{bib:AFS2019}.
\begin{definition}
  The canonical axis path of order $d$, denoted by $\axis^d:[0,1]\rightarrow\R^d$, is the unique piecewise linear path with $d$ segments whose graph interpolates the $d+1$ equidistant support points $(\frac{i}{d}, \sum_{j=1}^ie_j)$ with $0 \leq i \leq d$. 
\end{definition}
\begin{example}\label{ex:axis3}
For $d=3$, we have  
$$\axis^3(t)=
\begin{cases}
3t\mathrm{e}_1
&
0\leq t<\frac13,
\\
\mathrm{e}_1+3(t-\frac13)\mathrm{e}_2
&\frac13\leq t< \frac23,
\\
\mathrm{e}_1+\mathrm{e}_2+3(t-\frac23)\mathrm{e}_3
&\frac23\leq t\leq 1,
\end{cases}$$ 
and its $2$-truncated signature is $$\sigma(\mathsf{Axis}^3)=
1
\boldsymbol{\oplus}
\begin{bsmallmatrix}1\\1\\1\end{bsmallmatrix}
\boldsymbol{\oplus}
\begin{bsmallmatrix}\frac12&1&1\\0&\frac12&1\\0&0&\frac12\end{bsmallmatrix}\in\mathcal{G}_{3,2}.$$
\end{example}

Since integration is linear, it is easy to see (e.g. \cite[Lemma 2.1]{bib:PSS2019}) that the signature is equivariant. 
\begin{lemma}\label{lem:equivariance_signature}
For all $X:[0,1]\rightarrow\R^m$ and $A\in\R^{d\times m}$,
\begin{equation}\label{eq:equivariance_sig}
\sigma(A\cdot X)=A\cdot\sigma(X),
\end{equation}
where $A\cdot X:[0,1]\rightarrow\R^d,t\mapsto A\cdot X(t)$ is the \emph{linearly transformed axis path}, and $A$ acts on $\sigma(X)$ via the matrix-tensor congruence operation \eqref{eq:def_tensorMatrixCongruence}.
\end{lemma}

Clearly, we can write every piecewise linear path $X:[0,1]\rightarrow \R^d$ with $m$ segments as a linearly transformed axis path $X=A\cdot\axis^m$ for a suitable matrix $A\in\R^{d\times m}$. In fact, the (cumulative) columns of $A$ are precisely those points in $X$ that are allowed to be nondifferentiable.  With equivariance \eqref{eq:equivariance_sig} we can therefore evaluate every signature of a piecewise linear path via the signature of the axis path.

\section{Path recovery}\label{sec:path_recovery}

For every $d,k$ and $m \geq 1$ we denote the image of the ($k$-truncated) signature, when restricted to piecewise linear paths with $m$ segments, by
\begin{align*}\mathcal{L}^{\mathsf{im}}_{d,\leq k,m}
&:=\left\{\sigma(X)\;\begin{array}{|l}X:[0,1]\rightarrow\R^{d}\\\text{piecewise linear}\\\text{with $m$ segments}\end{array}\right\}.
\end{align*} 
Note that $\mathcal{L}^{\mathsf{im}}_{d,k,m}$ is the projection to $k$-tensors introduced by \cite[Equation (37)]{bib:AFS2019}. For simplicity, we work with the entire sequence of tensors. The reason for this is that the projection from the Lie group to the $k$-th graded tensor component is only finite-to-one (see \cite[Theorem 6.1]{bib:AFS2019}) and the barycenter map requires the lower graded components in its argument. With equivariance \eqref{eq:equivariance_sig} on each tensor component we obtain 
$$\mathcal{L}^{\mathsf{im}}_{d,\leq k,m}=\left\{A\cdot \sigma(\mathsf{Axis}^{m})\mid A\in \R^{d\times m}\right\}.$$

\begin{corollary}
For all $\mathbf{s}\in T_{d,k}$ we have $\mathbf{s}\in\mathcal{L}^{\mathsf{im}}_{d,\leq k,m}$ if and only if the polynomial system \begin{equation}\label{eq:pol_system}\mathbf{s}=A\cdot \sigma(\axis^m)\end{equation} has a solution $A\in\R^{d\times m}$. In this case, we call $A\cdot\axis^m$ the \emph{recovered path} from $\mathbf{s}$.  
\end{corollary}
\begin{proof}We use \cite[Section 2]{bib:PSS2019} on every tensor component. 
\end{proof}

Given signatures of $N$ generic piecewise linear paths $X_{i}:[0,1]\rightarrow\R^{d}$ with $\alpha_i$ segments, we are interested in piecewise linear path recovery from the Lie group barycenter of $\left(\sigma(X_1),\dots,\sigma(X_N)\right)$ with a minimal number of segments.

\begin{example}\label{ex:lowerBoundB2211}
Consider the two $1$-segment paths $X_1,X_2:[0,1]\rightarrow\R^2$ given by $X_1(t):=t\begin{bsmallmatrix}
    1\\1/2
\end{bsmallmatrix}$ and $X_2(t):=t\begin{bsmallmatrix}
    1/2\\1
\end{bsmallmatrix}$ for $t\in[0,1]$. The $2$-truncated signature of a $1$-segment path in $\R^2$ is 
$$\begin{bmatrix}a_1\\a_2\end{bmatrix}\cdot\sigma\left(\axis^1\right)
=
1
\boldsymbol{\oplus}
\begin{bmatrix}a_1\\a_2\end{bmatrix}
\boldsymbol{\oplus}
\frac12\begin{bmatrix}a_1^2&a_1a_2\\a_1a_2&a_2^2\end{bmatrix}$$
where $a_1,a_2\in\R$. Furthermore, we have
$\bary(\sigma(X_1),\sigma(X_2))
=1
\boldsymbol{\oplus}
\frac34\begin{bsmallmatrix}1\\1\end{bsmallmatrix}
\boldsymbol{\oplus}
\frac9{32}\begin{bsmallmatrix}1&1\\1&1\end{bsmallmatrix}$ using, for instance, the closed-form expression of barycenters from \Cref{closed_formula_bary_k2}. If $m=1$, then the associated system \eqref{eq:pol_system} yields exactly one $1$-segment path solution $Y:[0,1]\rightarrow\R^2,t\mapsto \frac34t\begin{bsmallmatrix}
    1\\1
\end{bsmallmatrix}$ for the barycenter of $\sigma(X_1)$ and $\sigma(X_2)$. We illustrate the full verification with our \texttt{OSCAR} implementation in \Cref{ex:lowerBoundB2211_OSCAR}.
\end{example}

\begin{definition} For fixed $d,k$ let 
    \begin{align*}
    B_{d,k}(\alpha):=\min\Big\{m\in\N\mid
    \mathsf{bary}\left(\mathcal{L}^{\mathrm{im}}_{d,\leq k,\alpha_1}\times\dots\times\mathcal{L}^{\mathrm{im}}_{d,\leq k,\alpha_N}\right)\subseteq\mathcal{L}^{\mathrm{im}}_{d,\leq k,m}\Big\}
\end{align*}
be the \emph{barycenter recovery order} of $\alpha=(\alpha_1,\dots,\alpha_N)\in\N^N$. 
\end{definition}

\begin{example}
We illustrate several underlying recovered paths to determine $B_{d,k}(\alpha)$ with $d=2$ in \Cref{fig:Bdkalpha}. Note that for $d=2$ the signature stores the distance and the signed area of a path, see e.g.  \cite{chevyrev2025primersignaturemethodmachine}. We provide a lower bound of $B_{2,2}(1,1)$ in \Cref{ex:lowerBoundB2211}. The explicit verifications of $B_{2,k}$ from \Cref{fig:Bdkalpha} are postponed to \Cref{thm:main_matrix} and \Cref{prop:B2311}. 
\end{example}

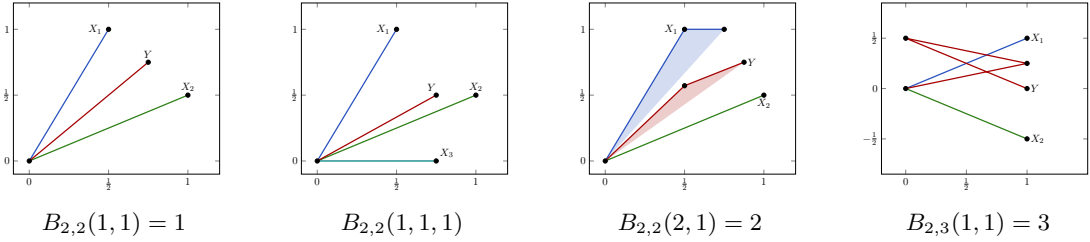
\begin{figure}[h]
    \centering
    \begin{subfigure}{0.24\textwidth}
        \centering
\begin{tikzpicture}[scale=0.4]
\begin{axis}[
    xtick={0,1/2,1},
    xticklabels={$0$,$\frac{1}{2}$,$1$}, 
    ytick={0,1/2,1},
    yticklabels={$0$,$\frac12$,$1$},
    xmin=-0.1, xmax=1.2,
    ymin=-0.1, ymax=1.2,
    domain=0:1,
]

    % Non-constant part in red (linear interpolation)
    \addplot[CeruleanBlue, very thick] coordinates {(0,0) (1/2,1)};

    \addplot[myGreen, very thick] coordinates {(0,0) (1,1/2)};

    \addplot[darkred, very thick] coordinates {(0,0) (3/4,3/4)};

   % Marking the points
    \addplot[mark=*, only marks] coordinates {(0,0) (1/2,1) (3/4,3/4) (1,1/2)};

\node[above] at (axis cs:3/4,3/4) {$Y$};
        \node[above] at (axis cs:1,1/2) {$X_2$};
    \node[left] at (axis cs:1/2,1) {$X_1\;$};

\end{axis}
\end{tikzpicture}
\caption*{$B_{2,2}(1,1)=1$}
\end{subfigure}
\hfill
        \begin{subfigure}{0.24\textwidth}
        \centering
\begin{tikzpicture}[scale=0.4]
\begin{axis}[
    xtick={0,1/2,1},
    xticklabels={$0$,$\frac{1}{2}$,$1$}, 
    ytick={0,1/2,1},
    yticklabels={$0$,$\frac12$,$1$},
    xmin=-0.1, xmax=1.2,
    ymin=-0.1, ymax=1.2,
    domain=0:1,
]

    % Non-constant part in red (linear interpolation)
    \addplot[CeruleanBlue, very thick] coordinates {(0,0) (1/2,1)};

    \addplot[myGreen, very thick] coordinates {(0,0) (1,1/2)};
    
    \addplot[teal, very thick] coordinates {(0,0) (3/4,0)};

    \addplot[darkred, very thick] coordinates {(0,0) (3/4,1/2)};

   % Marking the points
    \addplot[mark=*, only marks] coordinates {(0,0) (1/2,1) (3/4,1/2) (1,1/2) (3/4,0)};

\node[above left] at (axis cs:3/4,1/2) {$\;\;\;Y\!\!$};
    \node[above] at (axis cs:1,1/2) {$X_2$};
    \node[left] at (axis cs:1/2,1) {$X_1\;$};
    \node[above right] at (axis cs:3/4,0) {$X_3$};

\end{axis}
\end{tikzpicture}
\caption*{$B_{2,2}(1,1,1)$}
\end{subfigure}
\hfill
\begin{subfigure}{0.24\textwidth}
        \centering
\begin{tikzpicture}[scale=0.4]
\begin{axis}[
    xtick={0,1/2,1},
    xticklabels={$0$,$\frac{1}{2}$,$1$},  
    ytick={0,1/2,1},
    yticklabels={$0$,$\frac12$,$1$},
    xmin=-0.1, xmax=1.2,
    ymin=-0.1, ymax=1.2,
    domain=0:1,
]

    % Non-constant part in red (linear interpolation)
    \addplot[CeruleanBlue, very thick] coordinates {(0,0) (1/2,1) (3/4,1)};

    \addplot[myGreen, very thick] coordinates {(0,0) (1,1/2)};

      \addplot[darkred, very thick] coordinates {(0,0) (1/2,4/7) (1/2+3/8,4/7+5/28)};

\path[fill=CeruleanBlue,opacity=0.2] 
    (axis cs:0,0) -- (axis cs:1/2,1) -- (axis cs:3/4,1) -- cycle;

\path[fill=darkred,opacity=0.2] 
    (axis cs:0,0) -- (axis cs:1/2,4/7) -- (axis cs:1/2+3/8,4/7+5/28) -- cycle;

   % Marking the points
    \addplot[mark=*, only marks] coordinates {(0,0) (1/2,1)  (1,1/2) (3/4,1) (1/2,4/7) (1/2+3/8,4/7+5/28)};

        \node[below] at (axis cs:1,1/2) {$X_2$};
    \node[left] at (axis cs:1/2,1) {$X_1\;$};
    \node[right] at (axis cs:1/2+3/8,4/7+5/28) {$Y$};

\end{axis}
\end{tikzpicture}
\caption*{$B_{2,2}(2,1)=2$}
\end{subfigure}
\hfill
\begin{subfigure}{0.24\textwidth}
        \centering
\begin{tikzpicture}[scale=0.4]
\begin{axis}[
   xtick={0,1/2,1},
    xticklabels={$0$,$\frac{1}{2}$,$1$}, 
    ytick={-1/2,0,1/2},
    yticklabels={$-\frac12$,$0$,$\frac12$},
    xmin=-0.2, xmax=1.5,
    ymin=-0.85, ymax=0.85,
    domain=0:1,
]

    % Non-constant part in red (linear interpolation)
    \addplot[CeruleanBlue, very thick] coordinates {(0,0) (1,1/2)};

        \addplot[myGreen, very thick] coordinates {(0,0) (1,-1/2)};

    \addplot[darkred, very thick] coordinates {(0,0) (1,1/4) (0,1/2) (1,0)};

   % Marking the points
    \addplot[mark=*, only marks] coordinates {(0,0) (1,1/4) (0,1/2) (1,0) (1,-1/2) (1,1/2)};

\node[right] at (axis cs:1,0) {$Y$};
\node[right] at (axis cs:1,1/2) {$X_1$};
\node[right] at (axis cs:1,-1/2) {$X_2$};

\end{axis}
\end{tikzpicture}
\caption*{$B_{2,3}(1,1)=3$}
\end{subfigure}
%\caption{Illustrations of $Y$ with $\mathsf{bary}(\sigma(X_1),\dots,\sigma(X_N))=\sigma(Y)$. 
%}
\caption{Let $X_1,\dots,X_N$ be a sample of paths with $\alpha\in\N^N$ segments and $Y$ be a path with the minimal number of segments $B_{2,k}(\alpha)$ such that $\bary(\sigma(X_1),\dots,\sigma(X_N))=\sigma(Y)$. 
On the left, we provide illustrations of $N=2,3$ and quadratic truncation,  where the distances and  signed areas are averaged. On the right, we illustrate a minimal $3$-segment recovery with cubic truncation.}
\label{fig:Bdkalpha}
\end{figure}

\begin{proposition}\label{prop:well_def_bary_order}
The number $B_{d,k}(\alpha)$ is well-defined, that is, the minimum is attained. 
\end{proposition}

In order to prove \Cref{prop:well_def_bary_order}, we recall that every element of the free nilpotent Lie group can be realized as the signature of a piecewise linear path. This result is known as Chen-Chow's theorem, see for instance \cite[Theorem 7.28]{Friz_Victoir_2010}. 

\begin{theorem}[Chen-Chow]\label{thm:ChenChow}
For every $\mathbf{z}\in\mathcal{G}_{d,k}$ there exists an $m\in\N$ such that $\mathbf{z}\in\mathcal{L}^{\mathsf{im}}_{d,\leq k,m}$. 
\end{theorem}

The key idea is to reduce the problem to a single path recovery. For this, we introduce a certain path decomposition of the axis path. Note that $B_{d,k}(\alpha)$ is invariant under permuting the entries of the argument $\alpha$, using \Cref{lem:bary_idempotent}. However, since we want to restrict to a single path, the following construction depends on the order of entries in $\alpha$. Therefore, we move to the set of \emph{compositions} for $m\in\N$, denoted by $\mathcal{C}_m$.

\begin{definition}\label{def:canonicalAxisSubpath}
Let $E^j:[0,1]\rightarrow\R^m,t\mapsto t\mathrm{e}_j$ denote the $j$-th \emph{canonical segment path} in $\R^m$. For any $\alpha\in\mathcal{C}_{m}$ of length $N$ 
    and $i \in[N]$ we define the $i$-th \emph{axis subpath}
    via path concatenation of $\alpha_i$ canonical segment paths,  
    \begin{equation}\label{eq:defAxSub}\mathsf{Ax}^{\alpha,i}:=E^{\alpha_1+\dots+\alpha_{i-1}+1}\star\dots\star E^{\alpha_1+\dots+\alpha_{i}}:[0,1]\rightarrow\R^m.\end{equation}
\end{definition}

\begin{lemma}\label{lem:axis_dec}For all $\alpha\in\mathcal{C}_{m}$ of length $N$, 
\begin{enumerate}[i)]
\item $\mathsf{Ax}^{\alpha,i}$ is a subpath of $\axis^{m}$, and  
\item $\axis^{m}=\mathsf{Ax}^{\alpha,1}\star\dots\star\mathsf{Ax}^{\alpha,N}$. 
\end{enumerate}
\end{lemma}
We illustrate all axis subpaths for $m=3$ in \Cref{fig:axis3_compositions}. 

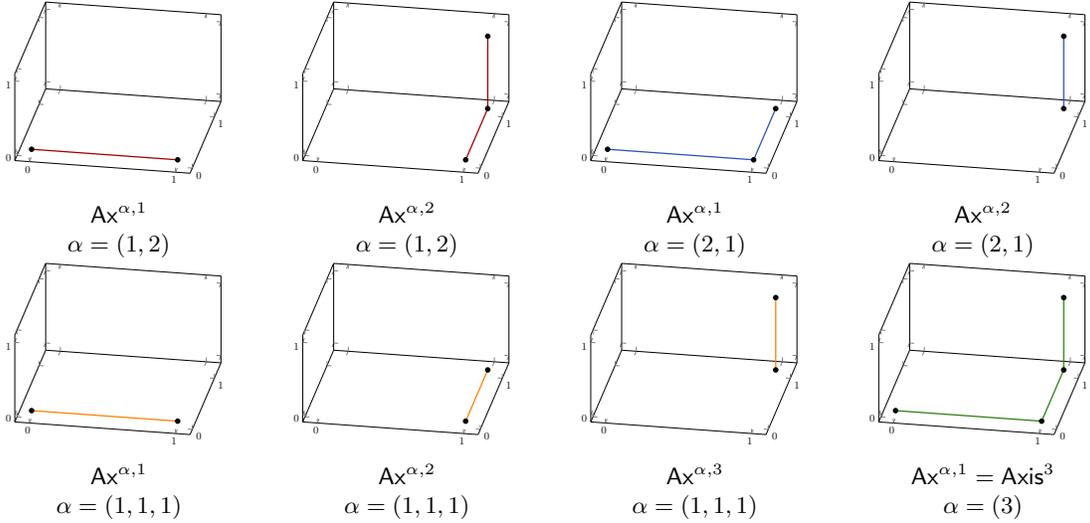
\begin{figure}[h]
    \centering
    \begin{subfigure}{0.24\textwidth}
        \centering
\begin{tikzpicture}[scale=0.4]
\begin{axis}[
    view={10}{40},              % View angle: azimuth and elevation
%    axis lines=middle,            % Axis lines through the middle
%    xlabel={$x$},
%    ylabel={$y$},
%    zlabel={$z$},
    xtick={0,1},
    ytick={0,1},
    ztick={0,1},
    xmin=-0.1, xmax=1.1,             % Extend axes slightly beyond the last point
    ymin=-0.1, ymax=1.3,
    zmin=-0.1, zmax=1.1,
%    width=10cm,                   % Width of the plot
%    height=8cm,                   % Height of the plot
%    grid=major,                   % Adds grid lines
]

    % Piecewise linear line in 3D
    \addplot3[
        very thick, darkred,                % Line style: 
        mark=*,
        mark options={black} % Marker style
    ]
    coordinates {
        (0,0,0)
        (1,0,0)
    };

\end{axis}
\end{tikzpicture}
\caption*{$\mathsf{Ax}^{\alpha,1}$\\ 
$\alpha=(1,2)$}
\end{subfigure}
\hfill
        \begin{subfigure}{0.24\textwidth}
        \centering
\begin{tikzpicture}[scale=0.4]
\begin{axis}[
    view={10}{40},              % View angle: azimuth and elevation
%    axis lines=middle,            % Axis lines through the middle
%    xlabel={$x$},
%    ylabel={$y$},
%    zlabel={$z$},
    xtick={0,1},
    ytick={0,1},
    ztick={0,1},
    xmin=-0.1, xmax=1.1,             % Extend axes slightly beyond the last point
    ymin=-0.1, ymax=1.3,
    zmin=-0.1, zmax=1.1,
%    width=10cm,                   % Width of the plot
%    height=8cm,                   % Height of the plot
%    grid=major,                   % Adds grid lines
]

    % Piecewise linear line in 3D
    \addplot3[
        very thick, darkred,                % Line style: 
        mark=*,
        mark options={black} % Marker style
    ]
    coordinates {
        (1,0,0)
        (1,1,0)
        (1,1,1)
    };

\end{axis}
\end{tikzpicture}
\caption*{$\mathsf{Ax}^{\alpha,2}$\\ 
$\alpha=(1,2)$}
\end{subfigure}
\hfill
\begin{subfigure}{0.24\textwidth}
        \centering
\begin{tikzpicture}[scale=0.4]
\begin{axis}[
    view={10}{40},              % View angle: azimuth and elevation
%    axis lines=middle,            % Axis lines through the middle
%    xlabel={$x$},
%    ylabel={$y$},
%    zlabel={$z$},
    xtick={0,1},
    ytick={0,1},
    ztick={0,1},
    xmin=-0.1, xmax=1.1,             % Extend axes slightly beyond the last point
    ymin=-0.1, ymax=1.3,
    zmin=-0.1, zmax=1.1,
%    width=10cm,                   % Width of the plot
%    height=8cm,                   % Height of the plot
%    grid=major,                   % Adds grid lines
]

    % Piecewise linear line in 3D
    \addplot3[
        very thick, CeruleanBlue,                % Line style: 
        mark=*,
        mark options={black} % Marker style
    ]
    coordinates {
        (0,0,0)
        (1,0,0)
        (1,1,0)
    };

\end{axis}
\end{tikzpicture}
\caption*{$\mathsf{Ax}^{\alpha,1}$\\ 
$\alpha=(2,1)$}
\end{subfigure}
\hfill
\begin{subfigure}{0.24\textwidth}
        \centering
\begin{tikzpicture}[scale=0.4]
\begin{axis}[
    view={10}{40},              % View angle: azimuth and elevation
%    axis lines=middle,            % Axis lines through the middle
%    xlabel={$x$},
%    ylabel={$y$},
%    zlabel={$z$},
    xtick={0,1},
    ytick={0,1},
    ztick={0,1},
    xmin=-0.1, xmax=1.1,             % Extend axes slightly beyond the last point
    ymin=-0.1, ymax=1.3,
    zmin=-0.1, zmax=1.1,
%    width=10cm,                   % Width of the plot
%    height=8cm,                   % Height of the plot
%    grid=major,                   % Adds grid lines
]

    % Piecewise linear line in 3D
    \addplot3[
        very thick, CeruleanBlue,                % Line style: 
        mark=*,
        mark options={black} % Marker style
    ]
    coordinates {
        (1,1,0)
        (1,1,1)
    };

\end{axis}
\end{tikzpicture}
\caption*{$\mathsf{Ax}^{\alpha,2}$\\ 
$\alpha=(2,1)$}
\end{subfigure}
    \begin{subfigure}{0.24\textwidth}
        \centering
\begin{tikzpicture}[scale=0.4]
\begin{axis}[
    view={10}{40},              % View angle: azimuth and elevation
%    axis lines=middle,            % Axis lines through the middle
%    xlabel={$x$},
%    ylabel={$y$},
%    zlabel={$z$},
    xtick={0,1},
    ytick={0,1},
    ztick={0,1},
    xmin=-0.1, xmax=1.1,             % Extend axes slightly beyond the last point
    ymin=-0.1, ymax=1.3,
    zmin=-0.1, zmax=1.1,
%    width=10cm,                   % Width of the plot
%    height=8cm,                   % Height of the plot
%    grid=major,                   % Adds grid lines
]

    % Piecewise linear line in 3D
    \addplot3[
        very thick, orange,                % Line style: 
        mark=*,
        mark options={black} % Marker style
    ]
    coordinates {
        (0,0,0)
        (1,0,0)
    };

\end{axis}
\end{tikzpicture}
\caption*{$\mathsf{Ax}^{\alpha,1}$\\ 
$\alpha=(1,1,1)$}
\end{subfigure}
\hfill
        \begin{subfigure}{0.24\textwidth}
        \centering
\begin{tikzpicture}[scale=0.4]
\begin{axis}[
    view={10}{40},              % View angle: azimuth and elevation
%    axis lines=middle,            % Axis lines through the middle
%    xlabel={$x$},
%    ylabel={$y$},
%    zlabel={$z$},
    xtick={0,1},
    ytick={0,1},
    ztick={0,1},
    xmin=-0.1, xmax=1.1,             % Extend axes slightly beyond the last point
    ymin=-0.1, ymax=1.3,
    zmin=-0.1, zmax=1.1,
%    width=10cm,                   % Width of the plot
%    height=8cm,                   % Height of the plot
%    grid=major,                   % Adds grid lines
]

    % Piecewise linear line in 3D
    \addplot3[
        very thick, orange,                % Line style: 
        mark=*,
        mark options={black} % Marker style
    ]
    coordinates {
        (1,0,0)
        (1,1,0)
    };

\end{axis}
\end{tikzpicture}
\caption*{$\mathsf{Ax}^{\alpha,2}$\\ 
$\alpha=(1,1,1)$}
\end{subfigure}
\hfill
\begin{subfigure}{0.24\textwidth}
        \centering
\begin{tikzpicture}[scale=0.4]
\begin{axis}[
    view={10}{40},              % View angle: azimuth and elevation
%    axis lines=middle,            % Axis lines through the middle
%    xlabel={$x$},
%    ylabel={$y$},
%    zlabel={$z$},
    xtick={0,1},
    ytick={0,1},
    ztick={0,1},
    xmin=-0.1, xmax=1.1,             % Extend axes slightly beyond the last point
    ymin=-0.1, ymax=1.3,
    zmin=-0.1, zmax=1.1,
%    width=10cm,                   % Width of the plot
%    height=8cm,                   % Height of the plot
%    grid=major,                   % Adds grid lines
]

    % Piecewise linear line in 3D
    \addplot3[
        very thick, orange,                % Line style: 
        mark=*,
        mark options={black} % Marker style
    ]
    coordinates {
        (1,1,0)
        (1,1,1)
    };

\end{axis}
\end{tikzpicture}
\caption*{$\mathsf{Ax}^{\alpha,3}$\\ 
$\alpha=(1,1,1)$}
\end{subfigure}
\hfill
\begin{subfigure}{0.24\textwidth}
        \centering
\begin{tikzpicture}[scale=0.4]
\begin{axis}[
    view={10}{40},              % View angle: azimuth and elevation
%    axis lines=middle,            % Axis lines through the middle
%    xlabel={$x$},
%    ylabel={$y$},
%    zlabel={$z$},
    xtick={0,1},
    ytick={0,1},
    ztick={0,1},
    xmin=-0.1, xmax=1.1,             % Extend axes slightly beyond the last point
    ymin=-0.1, ymax=1.3,
    zmin=-0.1, zmax=1.1,
%    width=10cm,                   % Width of the plot
%    height=8cm,                   % Height of the plot
%    grid=major,                   % Adds grid lines
]

    % Piecewise linear line in 3D
    \addplot3[
        very thick, myGreen,                % Line style: 
        mark=*,
        mark options={black} % Marker style
    ]
    coordinates {
        (0,0,0)
        (1,0,0)
        (1,1,0)
        (1,1,1)
    };

\end{axis}
\end{tikzpicture}
\caption*{$\mathsf{Ax}^{\alpha,1}=\axis^3$\\ 
$\alpha=(3)$}
\end{subfigure}
\caption{All axis subpaths $\mathsf{Ax}^{\alpha,i}$ of $\axis^3$ with respect to $\alpha\in\mathcal{C}_3$.}
\label{fig:axis3_compositions}
\end{figure}

\begin{proof}[Proof of \Cref{prop:well_def_bary_order}]
We provide an upper bound of $B_{d,k}(\alpha)$.
Let $m:=\alpha_1+\dots+\alpha_N$. By definition, there exists for every  $\mathbf{z}_i\in\mathcal{L}_{d,\leq k,\alpha_i}^{\mathsf{im}}$ an  $A_i\in\R^{d\times\alpha_i}$ such that $\mathbf{z}_i=\sigma(A_i\cdot\axis^{\alpha_i})$. We align these transformations column-wise, 
$$A:=\begin{bmatrix}A_1&\dots&A_N\end{bmatrix}\in\R^{d\times m}$$
and obtain 
$\mathbf{z}_i=\sigma(A\cdot\mathsf{Ax}^{\alpha,i})$ 
for every $i$, due to the invariance of signatures with respect to time reparametrization, \cite[Proposition~7.10]{Friz_Victoir_2010}. Therefore, 
\begin{align}\label{eq:equivarianceRed2OneRec}
\mathsf{bary}\left(
\mathbf{z}_1,\dots,\mathbf{z}_N\right)
&=
\mathsf{bary}\left(
\sigma(A\cdot\mathsf{Ax}^{\alpha,1}),\dots,\sigma(A\cdot\mathsf{Ax}^{\alpha,N})\right) \\
&=
A\cdot\mathsf{bary}\left(
\sigma(\mathsf{Ax}^{\alpha,1}),\dots,\sigma(\mathsf{Ax}^{\alpha,N})\right)\nonumber\end{align}
since the barycenter is a polynomial map (\Cref{lem:baryIsNonComPolyMap}) and both the signature and the evaluation of non-commutative polynomials is equivariant (\Cref{lem:evalEquivariant,lem:equivariance_signature}). Therefore, it suffices to recover one path from the Lie group element 
\begin{equation}\label{eq:mainTask}\mathsf{bary}\left(
\sigma(\mathsf{Ax}^{\alpha,1}),\dots,\sigma(\mathsf{Ax}^{\alpha,N})\right)=\sigma(U\cdot\axis^{u})\end{equation}
with $u\in\N$ and $U\in\R^{d\times u}$ chosen accordingly. For this last recovery, we used Chen-Chow's theorem, which we recalled in \Cref{thm:ChenChow}. Combining \eqref{eq:mainTask} and \eqref{eq:equivarianceRed2OneRec} we have $\bary(\mathbf{z}_1,\dots,\mathbf{z}_n)=\sigma(AU\cdot\axis^{u})$, where $u$ depends on the dimension $d$, the truncation level $k$ and the composition $\alpha$, but not on the explicit choice of $\mathbf{z}_i$. In particular, we conclude $B_{d,k}(\alpha)\leq u$. 
\end{proof}

\section{Matrix truncation}\label{sec:matrix_case}

In this section, we determine the barycenter recovery order $B_{d,k}(\alpha)$ for quadratic truncation level $k=2$. First, we show that the minimal recovered path for barycenters of segments is always a segment. For an illustration, see \Cref{fig:Bdkalpha}. 

\begin{proposition}\label{prop_B1111_k2}
For all $d$ and $N$, 
$$B_{d,2}(1,\dots,1)=1.$$
\end{proposition}
\begin{proof}
    We know that signatures of segments are precisely those where the log-signature is nonzero only at its first level; see \cite[Theorem 1.4]{FLS24} or \cite[Corollary 7.3]{AGOSS23}.  
If $k=2$ we can use the asymmetrized Baker--Campbell--Hausdorff formula and choose $f$ according to \eqref{eq:aBCHk2} in \Cref{rem:polynomialViaBCH}. Then, 
with \Cref{lem:baryIsNonComPolyMap}, 
$$\log\left(\bary\left(\left(\mathcal{L}_{d,\leq2,1}^{\mathrm{im}}\right)^N\right)\right)=\eval_f\left(\left(\log\left(\mathcal{L}_{d,\leq2,1}^{\mathrm{im}}\right)\right)^N\right)$$ is non-trivial only at its first level.
\end{proof}

We state our next main result, a formula for the barycenter recovery order in the case of quadratic truncation $k=2$.

\begin{theorem}\label{thm:main_matrix}
For all $\alpha\in\N^N$ and $d\geq 2$, 
$$B_{d,2}(\alpha)=\begin{cases}
\min(d,\alpha_+)&\text{all }\alpha_i\text{ even,}\\
\min(d,\alpha_+-\#_\alpha^{\mathsf{odd}}+1)&\text{otherwise,}
\end{cases}
$$
where $\alpha_+$ denotes the sum of all entries, and $\#_\alpha^{\mathsf{odd}}$ is the number of odd entries in $\alpha$.
\end{theorem}
We provide a proof in the remainder of this section. Note that \Cref{prop_B1111_k2} is a special case of \Cref{thm:main_matrix}. 
With \cite[Theorem 3.4]{bib:AFS2019} we immediately see that $B_{d,2}(\alpha)\leq d$. 
For the other bound dependent on $\alpha$, we enter the realm of matrix \emph{congruence normal forms}.  
From \cite[Equations (1) and (3)]{HORN20061010} we recall the canonical $2\times 2$ matrices 
\begin{equation} 
\Gamma_2:=
\begin{bmatrix}
0 & -1\\
1 & 1
\end{bmatrix}\quad\text{ and }\quad
\mathrm{H}_2(-1):=
\begin{bmatrix}
0 & 1\\
-1 & 0
\end{bmatrix}
\end{equation}
as the building blocks for our normal forms. Let $0_{c\times d}$ be the zero and $
\mathbbm{1}_{c\times d}$ the constant $1$ matrix in $\R^{c\times d}$. Furthermore, let $\mathrm{E}_{ij}$ be the matrix which is $1$ in the index $(i,j)$, and $0$ elsewhere.
Let $\mathrm{I}_m$ be the $m\times m$ identity matrix, and  
\begin{equation} 
\mathrm{U}_m:=\sum_{i=1}^m\sum_{j=i+1}^m\mathrm{E}_{ij}=
\begin{bsmallmatrix}
0 & 1 & \dots& 1\\
\vdots & \ddots & \ddots& \vdots\\
 &  & &1\\
0 & \dots &&0
\end{bsmallmatrix}
\end{equation}
the strictly upper triangular matrix with constant $1$ above the diagonal. We also use the notation 
$$M\oplus W:=
\begin{bmatrix}
M&0_{m\times w}\\
0_{w\times m}&W
\end{bmatrix}\in\R^{(m+w)\times(m+w)}
$$
for the block diagonal matrix of $M\in\R^{m\times m}$ and $W\in\R^{w\times w}$. Finally, let 
\begin{equation}\label{eq:defQm}\mathrm{Q}_m:=\bigoplus_{1\leq i\leq m}(-1)^{i+1}
 \end{equation}
 be the $m\times m$ matrix with alternating signs on its diagonal. Two matrices $M,V\in\R^{m\times m}$ are \emph{congruent}, denoted by $M\sim V$, if there exists an invertible $P\in\R^{m\times m}$ such that $PMP^\top=V$. 
We start with a preliminary consideration, an explicit congruence normal form according to \cite[Theorem 1]{HORN20061010}.

\begin{lemma}\label{lem:UmUT_nf}
For all $m$,
  $$\mathrm{U}_m-\mathrm{U}_m^\top\;\sim\;
  \begin{cases}
      \mathrm{H}_2(-1)^{\,\oplus \frac m2}& m\text{ even,}\\
     \mathrm{H}_2(-1)^{\,\oplus \frac {m-1}2}\oplus 0 & m\text{ odd.}
  \end{cases}
  $$
\end{lemma}
\begin{proof}
     For the base cases, we have equality, that is $\mathrm{U}_1-\mathrm{U}_1^\top=0$ and 
   $\mathrm{U}_2-\mathrm{U}_2^\top=\mathrm{H}_2(-1)$.
     Recursively, for all $1\leq i\leq \frac{m-1}2$, 
    \begin{align*}
&P_i\left(\mathrm{H}_2(-1)^{\oplus(i-1)}\oplus\left(\mathrm{U}_{m-2(i-1)}-\mathrm{U}_{m-2(i-1)}^\top\right)\right)P_i^\top
    =\mathrm{H}_2(-1)^{\oplus i}\oplus\left(\mathrm{U}_{m-2i}-\mathrm{U}_{m-2i}^\top\right)
\end{align*}
    where 
    $$P_i:=
    \mathrm{I}_m+\sum_{j=2i+1}^m \left(\mathrm{E}_{j,2i-1}-\mathrm{E}_{j,2i}\right)
    $$
    adds the $i$-th, and subtracts the $(i+1)$-th row to the $j$-th row for $i+2\leq j\leq m$, respectively. 
    In total, 
    $$P:=P_{\lfloor\frac{m-1}2\rfloor}\dots P_2P_1
    \;=\;
    \mathrm{I}_m+\sum_{i=1}^{\lfloor\frac{m-1}2\rfloor}\sum_{j=2i+1}^m\left(\mathrm{E}_{j,2i-1}-\mathrm{E}_{j,2i}\right)$$
    is the desired transform to the congruence normal form, $P\left(\mathrm{U}_m-\mathrm{U}_m^\top\right)P^\top$. 
\end{proof}

\begin{example}
  If $m=8$ then 
  $$\mathrm{U}_m-\mathrm{U}_m^\top
  =
  \begin{bsmallmatrix}
  0 & 1 & 1 & 1 & 1 & 1 & 1 & 1 \\
-1 & 0 & 1 & 1 & 1 & 1 & 1 & 1 \\
-1 & -1 & 0 & 1 & 1 & 1 & 1 & 1 \\
-1 & -1 & -1 & 0 & 1 & 1 & 1 & 1 \\
-1 & -1 & -1 & -1 & 0 & 1 & 1 & 1 \\
-1 & -1 & -1 & -1 & -1 & 0 & 1 & 1 \\
-1 & -1 & -1 & -1 & -1 & -1 & 0 & 1 \\
-1 & -1 & -1 & -1 & -1 & -1 & -1 & 0
\end{bsmallmatrix}
$$ 
  is congruent to 
  $$\mathrm{H}_2(-1)^{\oplus4}
  =
  \begin{bmatrix}
  0 & 1\\
  -1 & 0 
  \end{bmatrix}\oplus
  \begin{bmatrix}
  0 & 1\\
  -1 & 0 
  \end{bmatrix}\oplus
  \begin{bmatrix}
  0 & 1\\
  -1 & 0 
  \end{bmatrix}
  \oplus
  \begin{bmatrix}
  0 & 1\\
  -1 & 0 
  \end{bmatrix}
  =
  \begin{bsmallmatrix}
  0 & 1 & 0 & 0 & 0 & 0 & 0 & 0 \\
-1 & 0 & 0 & 0 & 0 & 0 & 0 & 0 \\
0 & 0 & 0 & 1 & 0 & 0 & 0 & 0 \\
0 & 0 & -1 & 0 & 0 & 0 & 0 & 0 \\
0 & 0 & 0 & 0 & 0 & 1 & 0 & 0 \\
0 & 0 & 0 & 0 & -1 & 0 & 0 & 0 \\
0 & 0 & 0 & 0 & 0 & 0 & 0 & 1 \\
0 & 0 & 0 & 0 & 0 & 0 & -1 & 0
\end{bsmallmatrix}
  $$ with the transformation $P=P_3P_2P_1$ given by
  $$P=
  \begin{bsmallmatrix}
1 & 0 & 0 & 0 & 0 & 0 & 0 & 0 \\
0 & 1 & 0 & 0 & 0 & 0 & 0 & 0 \\
1 & -1 & 1 & 0 & 0 & 0 & 0 & 0 \\
1 & -1 & 0 & 1 & 0 & 0 & 0 & 0 \\
1 & -1 & 0 & 0 & 1 & 0 & 0 & 0 \\
1 & -1 & 0 & 0 & 0 & 1 & 0 & 0 \\
1 & -1 & 0 & 0 & 0 & 0 & 1 & 0 \\
1 & -1 & 0 & 0 & 0 & 0 & 0 & 1 
  \end{bsmallmatrix}
  \,
   \begin{bsmallmatrix}
1 & 0 & 0 & 0 & 0 & 0 & 0 & 0 \\
0 & 1 & 0 & 0 & 0 & 0 & 0 & 0 \\
0 & 0 & 1 & 0 & 0 & 0 & 0 & 0 \\
0 & 0 & 0 & 1 & 0 & 0 & 0 & 0 \\
0 & 0 & 1 & -1 & 1 & 0 & 0 & 0 \\
0 & 0 & 1 & -1 & 0 & 1 & 0 & 0 \\
0 & 0 & 1 & -1 & 0 & 0 & 1 & 0 \\
0 & 0 & 1 & -1 & 0 & 0 & 0 & 1
  \end{bsmallmatrix}
  \,
  \begin{bsmallmatrix}
1 & 0 & 0 & 0 & 0 & 0 & 0 & 0 \\
0 & 1 & 0 & 0 & 0 & 0 & 0 & 0 \\
0 & 0 & 1 & 0 & 0 & 0 & 0 & 0 \\
0 & 0 & 0 & 1 & 0 & 0 & 0 & 0 \\
0 & 0 & 0 & 0 & 1 & 0 & 0 & 0 \\
0 & 0 & 0 & 0 & 0 & 1 & 0 & 0 \\
0 & 0 & 0 & 0 & 1 & -1 & 1 & 0 \\
0 & 0 & 0 & 0 & 1 & -1 & 0 & 1
  \end{bsmallmatrix} 
  =
  \begin{bsmallmatrix}
1 & 0 & 0 & 0 & 0 & 0 & 0 & 0 \\
0 & 1 & 0 & 0 & 0 & 0 & 0 & 0 \\
1 & -1 & 1 & 0 & 0 & 0 & 0 & 0 \\
1 & -1 & 0 & 1 & 0 & 0 & 0 & 0 \\
1 & -1 & 1 & -1 & 1 & 0 & 0 & 0 \\
1 & -1 & 1 & -1 & 0 & 1 & 0 & 0 \\
1 & -1 & 1 & -1 & 1 & -1 & 1 & 0 \\
1 & -1 & 1 & -1 & 1 & -1 & 0 & 1
\end{bsmallmatrix}
  $$ according to the proof of 
  \Cref{lem:UmUT_nf}. 
\end{example}

\begin{corollary}\label{cor:UdmUdT_invertible}
    If  $m$ is even, then $\mathrm{U}_m-\mathrm{U}_{m}^\top$ is invertible with inverse
    $${\left(\mathrm{U}_m-\mathrm{U}_{m}^\top\right)}^{-1}
    =
    \mathrm{Q}_m\left(\mathrm{U}_m-\mathrm{U}_{m}^\top\right)\mathrm{Q}_m.
    $$
\end{corollary}
\begin{proof}
  We use $\mathrm{H}_2(-1)^{-1}=\mathrm{H}_2(-1)^\top$ and $P(\mathrm{U}_m-\mathrm{U}_m)P^\top=\mathrm{H}_m(-1)^{\oplus\frac m2}$ with $P$ from the proof of \Cref{lem:UmUT_nf}.
\end{proof}

We provide the congruence normal form for the signature matrix of the axis path $\axis^m$. We recall from \cite[Example 2.1]{bib:AFS2019} or \cite[Equation (8)]{bib:PSS2019} the closed-form expression
\begin{equation}\label{eq:sigAxism}
C_m:=\sigma^{(2)}(\axis^m)=\frac12\,\mathrm{I}_m+\mathrm{U}_m
\end{equation}
for the canonical axis path in $m$-dimensional space. 

\begin{lemma}\label{lem:axismatrix_nf}For all $m$, 
    $$C_m
    \sim
    \begin{cases}
    \Gamma_2 \oplus\, \mathrm{H}_2(-1)^{\,\oplus \frac{m-2}{2}}& m\text{ even,}\\
     1\oplus\,\mathrm{H}_2(-1)^{\,\oplus \frac {m-1}2} & m\text{ odd.}
  \end{cases}
    $$
\end{lemma}
\begin{proof}
Similarly as in \Cref{lem:UmUT_nf}, we can construct the explicit transformation matrix 
\begin{align}\label{eq:transformAxisNF}
P&:=\sqrt{2}\,\mathrm{T}_{12}^{m+1}\left(\mathrm{I}_m-\sum_{s=2}^m\mathrm{E}_{s1}
  +
  \sum_{i=1}^{\lfloor\frac{m-1}2\rfloor}\sum_{j=1}^{m-2i}\left(
  \mathrm{E}_{j,m-2(i-1)}
  -\mathrm{E}_{j,m-2(i-1)-1}
  \right)\right)\nonumber
\end{align}
where $\mathrm{T}_{12}$ denotes the matrix that permutes the first and second row. Then, by construction,    
    $PC_m
    P^\top$ is in normal form. 
\end{proof}

For every $\alpha\in\mathcal{C}_m$ of length $N$ we consider the matrix projection of the barycenter \eqref{eq:mainTask},
$$W_\alpha:=\mathsf{bary}^{(2)}\left(
\sigma(\mathsf{Ax}^{\alpha,1}),\dots,\sigma(\mathsf{Ax}^{\alpha,N})\right).$$
We show that this matrix is a block diagonal matrix plus a rank-$1$ matrix, where the blocks are of the form discussed in \Cref{lem:UmUT_nf}. 

\begin{lemma}\label{lem:barycenterMatrixWmFormalua} 
For every $\alpha\in\mathcal{C}_m$ of length $N$, we have 
\begin{equation*}
W_\alpha
=
\frac1{2N}\bigoplus_{i=1}^N\left(\mathrm{U}_{\alpha_i}-\mathrm{U}_{\alpha_i}^\top\right)+\frac1{2N^2}{\mathbbm{1}}_{m\times m}.
\end{equation*}
\end{lemma}
\begin{example} For $m=12$ and $\alpha={(4,6,2)}\in\mathcal{C}_{12}$ of length $N=3$ we have
 \begin{align*}W_{(4,6,2)}&=
 \frac1{6}\begin{bsmallmatrix}
 0 & 1 & 1 & 1 & 0 & 0 & 0 & 0 & 0 & 0 & 0 & 0 \\
-1 & 0 & 1 & 1 & 0 & 0 & 0 & 0 & 0 & 0 & 0 & 0 \\
-1 & -1 & 0 & 1 & 0 & 0 & 0 & 0 & 0 & 0 & 0 & 0 \\
-1 & -1 & -1 & 0 & 0 & 0 & 0 & 0 & 0 & 0 & 0 & 0 \\
0 & 0 & 0 & 0 & 0 & 1 & 1 & 1 & 1 & 1 & 0 & 0 \\
0 & 0 & 0 & 0 & -1 & 0 & 1 & 1 & 1 & 1 & 0 & 0 \\
0 & 0 & 0 & 0 & -1 & -1 & 0 & 1 & 1 & 1 & 0 & 0 \\
0 & 0 & 0 & 0 & -1 & -1 & -1 & 0 & 1 & 1 & 0 & 0 \\
0 & 0 & 0 & 0 & -1 & -1 & -1 & -1 & 0 & 1 & 0 & 0 \\
0 & 0 & 0 & 0 & -1 & -1 & -1 & -1 & -1 & 0 & 0 & 0 \\
0 & 0 & 0 & 0 & 0 & 0 & 0 & 0 & 0 & 0 & 0 & 1 \\
0 & 0 & 0 & 0 & 0 & 0 & 0 & 0 & 0 & 0 & -1 & 0
\end{bsmallmatrix}
+\frac1{18}{\mathbbm{1}}_{12\times 12}.
\end{align*}
\end{example}
\begin{proof}[Proof of \Cref{lem:barycenterMatrixWmFormalua}]
Projection of \Cref{closed_formula_bary_k2} to matrix level yields  
\begin{equation}\label{eq:recallBary2ProofWm}
\bary^{(2)}(\mathbf{x})=\frac1N\sum_{i=1}^N\mathbf{x}_i^{(2)}-\frac1{2N}\sum_{i=1}^N\mathbf{x}_i^{(1)}\otimes\mathbf{x}_i^{(1)}
+\frac{1}{2N^2}\sum_{i_1=1}^N\sum_{i_2=1}^N\mathbf{x}_{i_1}^{(1)}\otimes\mathbf{x}_{i_2}^{(1)}
\end{equation}
for every sample $\mathbf{x}\in \mathcal{G}_{d,k}^N$. 
With equivariance of the signature and \Cref{lem:axis_dec}, we have 
\begin{equation*}
\sigma^{(2)}(\ax^{\alpha,i})
=0^{\oplus (\alpha_1+\dots+\alpha_{i-1})}\oplus\left(\frac12\,\mathrm{I}_{\alpha_i}+\mathrm{U}_{\alpha_i}\right)\oplus 0^{\oplus(\alpha_{i+1}+\dots+\alpha_{N})}
\end{equation*}
and therefore 
\begin{equation}\label{eq:proofWmmatrixterm}\frac1N\sum_{i=1}^N\sigma^{(2)}(\ax^{\alpha,i})
=
\frac1{2N}\mathrm{I}_{m}+\bigoplus_{i=1}^N\left(\mathrm{U}_{\alpha_i}-\mathrm{U}_{\alpha_i}^\top\right)
\end{equation}
for the first summand in \eqref{eq:recallBary2ProofWm}. The remaining two summands are algebraic combinations of the signature vectors
$$
\sigma^{(1)}(\ax^{\alpha,i})
=
\begin{bmatrix}0_{ \alpha_1+\dots+\alpha_{i-1}}\\\mathbbm{1}_{\alpha_i}\\0_{ m_{i+1}+\dots+\alpha_{N}}\end{bmatrix},
$$
that is 
\begin{equation}\label{eq:proofWm2ndsum}
    -\frac1{2N}\sum_{i=1}^N\sigma^{(1)}(\ax^{\alpha,i})\otimes\sigma^{(1)}(\ax^{\alpha,i})
    =
    -\frac1{2N}\bigoplus_{i=1}^N\mathbbm{1}_{\alpha_i\times\alpha_i}
\end{equation}
and 
\begin{equation}\label{eq:proofWm2ndsum2}
\frac{1}{2N^2}\sum_{i_1=1}^N\sum_{i_2=1}^N\sigma^{(1)}(\ax^{\alpha,i_1})\otimes\sigma^{(1)}(\ax^{\alpha,i_2})=\frac1{2N^2}\mathbbm{1}_{m\times m}
\end{equation}
for all $\alpha$. Combining \eqref{eq:proofWmmatrixterm}, \eqref{eq:proofWm2ndsum} and \eqref{eq:proofWm2ndsum2} in \eqref{eq:recallBary2ProofWm} yields the desired equality. 
\end{proof}

\begin{lemma}\label{cor:transposeWm}
Let $\alpha\in\mathcal{C}_m$ be of length $N$.  
\begin{enumerate}[(i)]
    \item\label{cor:transposeWm1} Transposition yields $W_\alpha^\top=-W_\alpha+\frac1{N^2}\mathbbm{1}_{m\times m}$.  
    \item\label{cor:transposeWm2} The symmetric part $W_\alpha^{\mathsf{sym}}=\frac1{2N^2}\mathbbm{1}_{m\times m}$ has rank $1$.  
     \item\label{cor:transposeWm3} The skew-symmetric part $$W_\alpha^{\mathsf{skew}}=\frac1{2N}\bigoplus_{i=1}^N\left(\mathrm{U}_{\alpha_i}-\mathrm{U}_{\alpha_i}^\top\right)$$ has rank $m-\#_\alpha^{\mathsf{odd}}$. 
    % \item\label{cor:transposeWm4} 
    % For the rank of $W_\alpha$ we abtain 
  %   $$m-\#_\alpha^{\mathsf{odd}}\leq \rank(W_\alpha)\leq m-\#_\alpha^{\mathsf{odd}}+1.$$\leonard{Can we use \cite[Lemma 1]{Meyer73} to obtain the rank? Then the proof of \Cref{lem:nf_Wm} should simplify drastically.}
\item \label{lem:WmInverseFormula}
    If $\alpha\in\mathcal{C}_m$ and all $\alpha_i$ are even, then $W_\alpha$ is invertible with inverse     $$W_\alpha^{-1}=4N^2\mathrm{Q}_mW_\alpha \mathrm{Q}_m.$$

    \end{enumerate}
\end{lemma}
\begin{proof}
Parts \eqref{cor:transposeWm1} and \eqref{cor:transposeWm2} follow from \Cref{lem:barycenterMatrixWmFormalua}. For part \eqref{cor:transposeWm3} we use the rank of every block from \Cref{lem:UmUT_nf}. 
For part \eqref{lem:WmInverseFormula} we note that $W_\alpha^{\mathsf{skew}}$ is invertible 
 as a block diagonal matrix of invertible matrices from \Cref{cor:UdmUdT_invertible}, 
$${(W_\alpha^{\mathsf{skew}})}^{-1}=2N\bigoplus_{i=1}^N{\left(\mathrm{U}_{\alpha_i}-\mathrm{U}_{\alpha_i}^\top\right)}^{-1}=2N\mathrm{Q}_m\left(\bigoplus_{i=1}^N{\mathrm{U}_{\alpha_i}-\mathrm{U}_{\alpha_i}^\top}\right)\mathrm{Q}_m=4N^2\mathrm{Q}_mW_\alpha^{\mathsf{skew}} \mathrm{Q}_m.$$
With \Cref{lem:barycenterMatrixWmFormalua}, 
$W_\alpha=W_\alpha^{\mathsf{skew}}+\frac1{2N^2}{\mathbbm{1}}_{m}{\mathbbm{1}}_{m}^\top$
and the second summand is a rank $1$ update with 
$$\frac1{2N^2}{\mathbbm{1}}_{m}^\top  W_\alpha^{\mathsf{skew}}{\mathbbm{1}}_{m}=0,$$
so $W_\alpha$ is invertible with 
inverse
$$W^{-1}_\alpha={(W_\alpha^{\mathsf{skew}})}^{-1}-\frac1{2N^2}{(W_\alpha^{\mathsf{skew}})}^{-1}{\mathbbm{1}}_{m}{\mathbbm{1}}_{m}^\top {(W_\alpha^{\mathsf{skew}})}^{-1}=2\mathrm{Q}_m(2N^2W_\alpha^{\mathsf{skew}}+{\mathbbm{1}}_{m\times m})\mathrm{Q}_m$$
using the Sherman–Morrison formula, e.g. \cite[eq. (160)]{petersen2008matrix}. 
\end{proof}

\begin{proposition}\label{lem:nf_Wm}For all $\alpha\in\mathcal{C}_m$, 
\begin{align*}
W_{\alpha}
\sim 
\begin{cases}
   \Gamma_2 \oplus\, \mathrm{H}_2(-1)^{\,\oplus \frac{m-2}{2}}&\text{all }\alpha_i\text{ even,}
   \\
1\oplus\mathrm{H}_2(-1)^{\,\oplus\frac12(m-\#_\alpha^{\mathsf{odd}})}\oplus 0^{\oplus(\#_\alpha^{\mathsf{odd}}-1)}
&\text{otherwise,}
\end{cases}
\end{align*}
where $\#^{\mathsf{odd}}_\alpha$ denotes the number of odd indices in $\alpha$.  
\end{proposition}

\begin{proof}
First, we assume that all $m_i$ are even.
From \cite[Lemma 2.1]{HORN20061010} we recall that nonsingular complex matrices are congruent, if and only if their cosquares are similar. If we find real transformation matrices with respect to similarity, then \cite[Theorem 1.1]{DOKOVIC2002149} guarantees the existence of real transformation matrices with respect to congruence.
With \Cref{cor:transposeWm}, 
$$W_\alpha^{-\top}W_\alpha=-4N^2\mathrm{Q}_mW_\alpha^\top \mathrm{Q}_m W_\alpha=-\mathrm{I}_m+\frac2{N}\mathrm{Q}_m\mathbbm{1}_{m\times m}$$
and analogously
$C_m^{-\top}C_m=-\mathrm{I}_m+2\mathrm{Q}_m\mathbbm{1}_{m\times m}$. Both can be transformed to the same Jordan normal form $1^{\oplus m-2}\oplus\mathrm{J}_2(-1)$ so the claim follows with 
\Cref{lem:axismatrix_nf}. 
\par
We continue with the case that at least one $\alpha_i$ is odd.
Note that $W_{(\alpha_1,\dots,\alpha_N)}
\sim W_{(\alpha_2,\dots,\alpha_N,\alpha_1)}$
with a suitable cyclic rotation of blocks, 
so we can assume without loss of generality that $\alpha_1$ odd. We apply
the congruence transform 
\begin{equation}\label{eq:transdormIfm1IsOddToBlock}
P_1:=\mathrm{I}_m+\sum_{i=\alpha_1+1}^m\sum_{j=1}^{\alpha_1}(-1)^j\,\mathrm{E}_{ij}\end{equation}
and obtain 
\begin{align*}
P_1W_\alpha P_1^\top&=\left(\frac1{2N}\left(\mathrm{U}_{\alpha_1}-\mathrm{U}_{\alpha_1}^\top\right)+\frac1{2N^2}\mathbbm{1}_{\alpha_1\times \alpha_1}\right)\oplus\frac1{2N}\bigoplus_{i=2}^N\left(\mathrm{U}_{\alpha_i}-\mathrm{U}_{\alpha_i}^\top\right)\\
&\sim 
\left(\mathrm{U}_{\alpha_1}-\mathrm{U}_{\alpha_1}^\top+\frac1{N}\mathbbm{1}_{\alpha_1\times \alpha_1}\right)
\oplus
\mathrm{H}_2(-1)^{\,\oplus\frac12(m-\alpha_1-\#_\alpha^{\mathsf{odd}}-1)}\oplus 0^{\oplus(\#_\alpha^{\mathsf{odd}}-1)}
\end{align*}
from \Cref{lem:UmUT_nf}.
It remains to consider
\begin{equation*}
Z:=\mathrm{U}_{\alpha_1}-\mathrm{U}_{\alpha_1}^\top+\frac1{N}\mathbbm{1}_{\alpha_1\times \alpha_1}
\end{equation*}
for odd $\alpha_1$. 
Analogous as above, $Z$ is invertible with  $Z^{-1}=\mathrm{Q}_{\alpha_1}(\mathrm{U}_{\alpha_1}-\mathrm{U}_{\alpha_1}^\top+N\mathbbm{1}_{\alpha_1\times \alpha_1})\mathrm{Q}_{\alpha_1}$ and the cosquare of $Z$ satisfies
$$Z^{-\top}Z=-\mathrm{I}_{\alpha_1}+2\mathrm{Q}_{\alpha_1}\mathbbm{1}_{\alpha_1\times \alpha_1}
=
C_{\alpha_1}^{-\top}C_{\alpha_1}
$$
so again with \cite[Lemma 2.1]{HORN20061010} and \cite[Theorem 1.1]{DOKOVIC2002149}, we get 
$Z\sim C_{\alpha_1}\sim 1\oplus \mathrm{H}_2(-1)^{\,\oplus\frac12(\alpha_1-1)}$
where the last congruence is according to \Cref{lem:axismatrix_nf}.
\end{proof}

\begin{corollary}The rank of $W_\alpha$ is $m$, if all $\alpha_i$ are even, and $m-\#^{\mathsf{odd}}_\alpha+1$, otherwise.
\end{corollary}

Note that \Cref{lem:nf_Wm} also implies \Cref{thm:main_matrix} when projected to signature matrices. However, since we consider path recovery over the entire Lie group (with even truncation level $k=2$) we cannot apply \cite[Proposition 6.3]{bib:AFS2019} for a one-to-one recovery via projection to signature matrices. For a one-to-one recovery, we provide the following explicit congruence transform.  

\begin{proof}[Alternative proof of \Cref{lem:nf_Wm}]
   We construct the congruence transform $P$ explicitly such that $PW_\alpha P^\top=C_{\rank(W_\alpha)}\oplus 0^{\oplus m-\rank(W_\alpha)}$. If all $\alpha_i$ are even, then 
\begin{equation}\label{eq:Q1WmToAxisd}
P_1:=\mathrm{I}_m-\left(\mathrm{U}^\top_m-\bigoplus_{i=1}^N\mathrm{U}^\top_{\alpha_i}\right)\mathrm{Q}_m
\end{equation}
transforms $W_\alpha$ to 
\begin{equation}\label{eq:matrixTransform}P_1W_\alpha P_1^\top=\frac1{2N^2}(\mathrm{I}_m+(N+1)\mathrm{U}_m-(N-1)\mathrm{U}_m^\top)\end{equation}
and with
$$P_2:=\mathrm{I}_m+(N-1)\sum_{j=1}^m\sum_{i=2\lfloor\frac j2\rfloor+1}^m\mathrm{E}_{ij}\mathrm{Q}_m$$
we obtain a linear transform 
such that 
$P_2P_1W_\alpha{(P_2P_1)}^\top=C_m$. 
\par
Otherwise, as seen in the first proof of \Cref{lem:nf_Wm}, we can assume that $\alpha_1$ is odd, and proceed analogously with the transform $P_1$ according to \eqref{eq:transdormIfm1IsOddToBlock}. We construct the transform 
$$P_2:=\left(N\mathrm{I}_{\alpha_1}+(N-1)\sum_{j=1}^{\frac{\alpha_1-1}2}\sum_{i=2j}^ {\alpha_1}\left(\mathrm{E}_{i,2j-1}-\mathrm{E}_{i,2j}\right)\right)\oplus\mathrm{I}_{\alpha_2+\dots+\alpha_N}$$
so that 
\begin{equation}\label{eq:prooftransformalpha1odd}
P_2P_1W_\alpha (P_2P_1)^\top=C_{\alpha_1}\oplus\frac1{2N}\bigoplus_{i=2}^N\left(\mathrm{U}_{\alpha_i}-\mathrm{U}_{\alpha_i}^\top\right)\end{equation}
for every $\alpha$. With \Cref{lem:UmUT_nf,lem:axismatrix_nf} we obtain a congruence transformation to the normal form, up to a reordering of elementary blocks. A last inverted transform according to \Cref{lem:axismatrix_nf} provides the desired congruence transform from $W_\alpha$ to $C_{\rank(W_\alpha)}\oplus 0^{\oplus m-\rank(W_\alpha)}$. 
\end{proof}

With this alternative proof, we can now prove the remaining main result. 

\begin{proof}[Proof of \Cref{thm:main_matrix}]Without loss of generality we can restrict to the case $\alpha\in\mathcal{C}_m$. 
The vector projection of \eqref{eq:mainTask} satisfies 
\begin{equation*}
\mathsf{bary}^{(1)}\left(
\sigma(\mathsf{Ax}^{\alpha,1}),\dots,\sigma(\mathsf{Ax}^{\alpha,N})\right)
=\frac1N\mathbbm{1}_{m}
\end{equation*}
where $N$ is the length of  $\alpha$.
We use the transformations according to the proof of \Cref{lem:nf_Wm}. 
\par
 If all $\alpha_i$ are even, then the matrix transforms $P_2$ and $P_1$ from \eqref{eq:Q1WmToAxisd} and \eqref{eq:matrixTransform} have the eigenvalues $1$ and $N$ for the same eigenvector $\mathbbm{1}_m$. Therefore, $P_2P_1$ simultaneously transforms the matrix and vector level of \eqref{eq:mainTask} to $\sigma(\axis^m)$, that is, $B_{d,k}(\alpha)= \rank(W_\alpha)$ according to the proof of \Cref{prop:well_def_bary_order}. 
\par
Otherwise, we assume that $\alpha_1$ is odd and proceed as in \eqref{eq:prooftransformalpha1odd}. Then $$\frac1NP_2P_1\mathbbm{1}_m=\begin{bmatrix}
    \mathbbm{1}_{\alpha_1}\\
    0_{\alpha_2+\dots+\alpha_N}
\end{bmatrix}$$
and since the matrix from \Cref{lem:axismatrix_nf} transforms $\mathbbm{1}_{\alpha_1}$ to $\mathrm{e}_1$ we see that \Cref{lem:nf_Wm} provides a compatible transform of vectors and matrices, i.e., $B_{d,k}(\alpha)= \rank(W_\alpha)$. 
\end{proof}

\begin{example}
We illustrate the construction of \Cref{thm:main_matrix} for $\alpha=(4,6,2)\in\mathcal{C}_{12}$. Since all $\alpha_i$ are even, we are in the first case of the construction and obtain our simultaneous transform 
$$
\begin{bsmallmatrix}
   3 & 0 & 0 & 0 & 0 & 0 & 0 & 0 & 0 & 0 & 0 & 0 \\
2 & 1 & 0 & 0 & 0 & 0 & 0 & 0 & 0 & 0 & 0 & 0 \\
2 & -2 & 3 & 0 & 0 & 0 & 0 & 0 & 0 & 0 & 0 & 0 \\
2 & -2 & 2 & 1 & 0 & 0 & 0 & 0 & 0 & 0 & 0 & 0 \\
2 & -2 & 2 & -2 & 3 & 0 & 0 & 0 & 0 & 0 & 0 & 0 \\
2 & -2 & 2 & -2 & 2 & 1 & 0 & 0 & 0 & 0 & 0 & 0 \\
2 & -2 & 2 & -2 & 2 & -2 & 3 & 0 & 0 & 0 & 0 & 0 \\
2 & -2 & 2 & -2 & 2 & -2 & 2 & 1 & 0 & 0 & 0 & 0 \\
2 & -2 & 2 & -2 & 2 & -2 & 2 & -2 & 3 & 0 & 0 & 0 \\
2 & -2 & 2 & -2 & 2 & -2 & 2 & -2 & 2 & 1 & 0 & 0 \\
2 & -2 & 2 & -2 & 2 & -2 & 2 & -2 & 2 & -2 & 3 & 0 \\
2 & -2 & 2 & -2 & 2 & -2 & 2 & -2 & 2 & -2 & 2 & 1 
\end{bsmallmatrix}
\,
\begin{bsmallmatrix}
    1 & 0 & 0 & 0 & 0 & 0 & 0 & 0 & 0 & 0 & 0 & 0 \\
0 & 1 & 0 & 0 & 0 & 0 & 0 & 0 & 0 & 0 & 0 & 0 \\
0 & 0 & 1 & 0 & 0 & 0 & 0 & 0 & 0 & 0 & 0 & 0 \\
0 & 0 & 0 & 1 & 0 & 0 & 0 & 0 & 0 & 0 & 0 & 0 \\
-1 & 1 & -1 & 1 & 1 & 0 & 0 & 0 & 0 & 0 & 0 & 0 \\
-1 & 1 & -1 & 1 & 0 & 1 & 0 & 0 & 0 & 0 & 0 & 0 \\
-1 & 1 & -1 & 1 & 0 & 0 & 1 & 0 & 0 & 0 & 0 & 0 \\
-1 & 1 & -1 & 1 & 0 & 0 & 0 & 1 & 0 & 0 & 0 & 0 \\
-1 & 1 & -1 & 1 & 0 & 0 & 0 & 0 & 1 & 0 & 0 & 0 \\
-1 & 1 & -1 & 1 & 0 & 0 & 0 & 0 & 0 & 1 & 0 & 0 \\
-1 & 1 & -1 & 1 & -1 & 1 & -1 & 1 & -1 & 1 & 1 & 0 \\
-1 & 1 & -1 & 1 & -1 & 1 & -1 & 1 & -1 & 1 & 0 & 1
\end{bsmallmatrix}
=
\begin{bsmallmatrix}
    3 & 0 & 0 & 0 & 0 & 0 & 0 & 0 & 0 & 0 & 0 & 0 \\
2 & 1 & 0 & 0 & 0 & 0 & 0 & 0 & 0 & 0 & 0 & 0 \\
2 & -2 & 3 & 0 & 0 & 0 & 0 & 0 & 0 & 0 & 0 & 0 \\
2 & -2 & 2 & 1 & 0 & 0 & 0 & 0 & 0 & 0 & 0 & 0 \\
-1 & 1 & -1 & 1 & 3 & 0 & 0 & 0 & 0 & 0 & 0 & 0 \\
-1 & 1 & -1 & 1 & 2 & 1 & 0 & 0 & 0 & 0 & 0 & 0 \\
-1 & 1 & -1 & 1 & 2 & -2 & 3 & 0 & 0 & 0 & 0 & 0 \\
-1 & 1 & -1 & 1 & 2 & -2 & 2 & 1 & 0 & 0 & 0 & 0 \\
-1 & 1 & -1 & 1 & 2 & -2 & 2 & -2 & 3 & 0 & 0 & 0 \\
-1 & 1 & -1 & 1 & 2 & -2 & 2 & -2 & 2 & 1 & 0 & 0 \\
-1 & 1 & -1 & 1 & -1 & 1 & -1 & 1 & -1 & 1 & 3 & 0 \\
-1 & 1 & -1 & 1 & -1 & 1 & -1 & 1 & -1 & 1 & 2 & 1
\end{bsmallmatrix}$$
from matrix level $W_\alpha$ to $C_{12}$, and vector level  $\frac13\mathbbm{1}_{12}$ to $\mathbbm{1}_{12}$, respectively. 
\end{example}

\begin{example}
    We illustrate the construction of \Cref{thm:main_matrix} for $\alpha=(5,4,3,4)\in\mathcal{C}_{16}$. Since $\alpha_1$ is odd, we use \eqref{eq:prooftransformalpha1odd} so that 
    $$P_2P_1=
    \begin{bsmallmatrix}
4 & 0 & 0 & 0 & 0 & 0_{1\times 11} \\
3 & 1 & 0 & 0 & 0 & 0_{1\times 11} \\
3 & -3 & 4 & 0 & 0 & 0_{1\times 11} \\
3 & -3 & 3 & 1 & 0 & 0_{1\times 11} \\
3 & -3 & 3 & -3 & 4 & 0_{1\times 11} \\
-{\mathbbm{1}}_{11} & {\mathbbm{1}}_{11} & -{\mathbbm{1}}_{11} & {\mathbbm{1}}_{11} & -{\mathbbm{1}}_{11} & \mathrm{I}_{11}
\end{bsmallmatrix}
    $$
    transforms $W_\alpha$ to 
    $C_5\oplus\frac18\left(\mathrm{U}_4-\mathrm{U}_4^\top\right)
    \oplus\frac18\left(\mathrm{U}_3-\mathrm{U}_3^\top\right)
    \oplus\frac18\left(\mathrm{U}_4-\mathrm{U}_4^\top\right)$.
    With \Cref{lem:UmUT_nf,lem:axismatrix_nf} we obtain a congruence transformation
    $$P_3:=
    \sqrt2
\begin{bsmallmatrix}
    1 & -1 & 1 & -1 & 1 \\
-1 & 1 & 0 & -1 & 1 \\
-1 & 0 & 1 & -1 & 1 \\
-1 & 0 & 0 & 1 & 0 \\
-1 & 0 & 0 & 0 & 1
\end{bsmallmatrix}
\oplus
2\sqrt2\begin{bsmallmatrix}
    1 & 0 & 0 & 0 \\
0 & 1 & 0 & 0 \\
1 & -1 & 1 & 0 \\
1 & -1 & 0 & 1
\end{bsmallmatrix}
\oplus2\sqrt2\begin{bsmallmatrix}
    1 & 0 & 0 \\
0 & 1 & 0 \\
1 & -1 & 1
\end{bsmallmatrix}
\oplus
2\sqrt2\begin{bsmallmatrix}
    1 & 0 & 0 & 0 \\
0 & 1 & 0 & 0 \\
1 & -1 & 1 & 0 \\
1 & -1 & 0 & 1
\end{bsmallmatrix}
$$
to $1\oplus \mathrm{H}_2(-1)^{\oplus 5}\oplus 0\oplus \mathrm{H}_2(-1)$. We permute the last two blocks and obtain our transformation matrix
$$V:=\begin{bmatrix}
    \mathrm{I}_{11}&0_{11}&0_{11\times 4}\\
    0_{4\times 11}&0_4&\mathrm{I}_4\\
    0_{1\times 11}&1&0_{1\times 4}
\end{bmatrix}\,P_3\,P_2\,P_1$$
to congruence normal form.  
Finally we 
perform an inverted transform according to \Cref{lem:axismatrix_nf}, 
and obtain our simultaneous transformation
$$
\frac1{\sqrt2}\begin{bsmallmatrix}
   1 & -1 & 1 & -1 & 1 & -1 & 1 & -1 & 1 & -1 & 1 & -1 & 1 & -1 & 1 & 0 \\
-1 & 1 & 0 & -1 & 1 & -1 & 1 & -1 & 1 & -1 & 1 & -1 & 1 & -1 & 1 & 0 \\
-1 & 0 & 1 & -1 & 1 & -1 & 1 & -1 & 1 & -1 & 1 & -1 & 1 & -1 & 1 & 0 \\
-1 & 0 & 0 & 1 & 0 & -1 & 1 & -1 & 1 & -1 & 1 & -1 & 1 & -1 & 1 & 0 \\
-1 & 0 & 0 & 0 & 1 & -1 & 1 & -1 & 1 & -1 & 1 & -1 & 1 & -1 & 1 & 0 \\
-1 & 0 & 0 & 0 & 0 & 1 & 0 & -1 & 1 & -1 & 1 & -1 & 1 & -1 & 1 & 0 \\
-1 & 0 & 0 & 0 & 0 & 0 & 1 & -1 & 1 & -1 & 1 & -1 & 1 & -1 & 1 & 0 \\
-1 & 0 & 0 & 0 & 0 & 0 & 0 & 1 & 0 & -1 & 1 & -1 & 1 & -1 & 1 & 0 \\
-1 & 0 & 0 & 0 & 0 & 0 & 0 & 0 & 1 & -1 & 1 & -1 & 1 & -1 & 1 & 0 \\
-1 & 0 & 0 & 0 & 0 & 0 & 0 & 0 & 0 & 1 & 0 & -1 & 1 & -1 & 1 & 0 \\
-1 & 0 & 0 & 0 & 0 & 0 & 0 & 0 & 0 & 0 & 1 & -1 & 1 & -1 & 1 & 0 \\
-1 & 0 & 0 & 0 & 0 & 0 & 0 & 0 & 0 & 0 & 0 & 1 & 0 & -1 & 1 & 0 \\
-1 & 0 & 0 & 0 & 0 & 0 & 0 & 0 & 0 & 0 & 0 & 0 & 1 & -1 & 1 & 0 \\
-1 & 0 & 0 & 0 & 0 & 0 & 0 & 0 & 0 & 0 & 0 & 0 & 0 & 1 & 0 & 0 \\
-1 & 0 & 0 & 0 & 0 & 0 & 0 & 0 & 0 & 0 & 0 & 0 & 0 & 0 & 1 & 0 \\
0 & 0 & 0 & 0 & 0 & 0 & 0 & 0 & 0 & 0 & 0 & 0 & 0 & 0 & 0 & 1
\end{bsmallmatrix}^{-1}
\!\!\!\!\!\!\cdot V
=
\begin{bsmallmatrix}
19 & -15 & 15 & -15 & 15 & -2 & -1 & 1 & -4 & 1 & -4 & 0 & -2 & -1 & 1 & -4 \\
33 & -29 & 30 & -30 & 30 & -4 & -2 & 2 & -8 & 2 & -8 & 0 & -4 & -2 & 2 & -8 \\
33 & -33 & 34 & -30 & 30 & -4 & -2 & 2 & -8 & 2 & -8 & 0 & -4 & -2 & 2 & -8 \\
33 & -33 & 33 & -29 & 30 & -4 & -2 & 2 & -8 & 2 & -8 & 0 & -4 & -2 & 2 & -8 \\
33 & -33 & 33 & -33 & 34 & -4 & -2 & 2 & -8 & 2 & -8 & 0 & -4 & -2 & 2 & -8 \\
30 & -26 & 26 & -26 & 26 & -4 & 2 & 2 & -8 & 2 & -8 & 0 & -4 & -2 & 2 & -8 \\
27 & -23 & 23 & -23 & 23 & -5 & 6 & 2 & -8 & 2 & -8 & 0 & -4 & -2 & 2 & -8 \\
27 & -23 & 23 & -23 & 23 & -1 & -2 & 2 & -4 & 2 & -8 & 0 & -4 & -2 & 2 & -8 \\
24 & -20 & 20 & -20 & 20 & 2 & -5 & 1 & 0 & 2 & -8 & 0 & -4 & -2 & 2 & -8 \\
24 & -20 & 20 & -20 & 20 & -2 & -1 & 1 & -4 & 2 & -4 & 0 & -4 & -2 & 2 & -8 \\
21 & -17 & 17 & -17 & 17 & -2 & -1 & 1 & -4 & 1 & 0 & 0 & -4 & -2 & 2 & -8 \\
21 & -17 & 17 & -17 & 17 & -2 & -1 & 1 & -4 & 1 & -4 & 0 & -4 & 2 & 2 & -8 \\
18 & -14 & 14 & -14 & 14 & -2 & -1 & 1 & -4 & 1 & -4 & 0 & -5 & 6 & 2 & -8 \\
18 & -14 & 14 & -14 & 14 & -2 & -1 & 1 & -4 & 1 & -4 & 0 & -1 & -2 & 2 & -4 \\
15 & -11 & 11 & -11 & 11 & -2 & -1 & 1 & -4 & 1 & -4 & 0 & 2 & -5 & 1 & 0 \\
-1 & 1 & -1 & 1 & -1 & 0 & 0 & 0 & 0 & 1 & -1 & 1 & 0 & 0 & 0 & 0 
\end{bsmallmatrix}
$$
from, respectively, matrix level $W_\alpha$ to $C_{15}\oplus 0$, and vector level  $\frac14\mathbbm{1}_{16}$ to $\begin{bsmallmatrix}
    \mathbbm{1}_{15}\\0
\end{bsmallmatrix}$. 
\end{example}

\section{Computational aspects}\label{sec:computational_aspects_oscar}
Efficient algorithms based on the structural properties of free Lie algebras for computing signatures were developed in the \texttt{python} package \cite{RG20}. Several frameworks for machine learning tasks are provided in \cite{kidger2021signatory} or \cite{diehl2024fruits}. Signature barycenters are efficiently computed using log-coordinates and symbolic preprocessing in \cite{clausel2024barycenterfreenilpotentlie}, based on BCH series and Gr\"obner bases of modules in \texttt{sage}. The recent \texttt{Macaulay2} package \cite{amendola2025computingpathsignaturevarieties} implements several parts of \cite{bib:AFS2019} to study applied algebraic geometry questions related to signatures.
\par
We provide a novel \texttt{OSCAR} implementation of Lie group barycenters and signatures for piecewise linear and polynomial paths.
 \texttt{OSCAR} \cite{OSCAR-book,OSCAR} is an innovative \emph{Open Source Computer Algebra Research} system that is written in \texttt{julia}. It brings together powerful tools and compatible constructions relevant to this project, including multivariate arrays, Lie theory, non-commutative polynomials, or theory on Gröbner bases. 
\par
 We introduce two new structs \texttt{TruncTensorSeq} and \texttt{TruncTensorSeqElem} that implement the algebra of truncated tensor sequences according to \eqref{eq:truncTensSeq}. Basic constructors such as \texttt{one} or \texttt{sig\_axis} provide signatures of standard paths. More involved examples use Chen's identity (\Cref{thm:Chen}) or equivariance (\Cref{lem:equivariance_signature}) via group multiplication or matrix-tensor congruence, respectively.  

 \begin{example}
We compute the $2$-truncated signature of a $2$-dimensional $2$-segment path. 
\begin{verbatim}
k = 2; d = 2; m = 2
R, a = polynomial_ring_sig_transform(d,m)     # polynomial ring QQ[a₁₁,a₂₁,a₁₂,a₂₂]
TTSm = trunc_tensor_seq(R,k,m)                # creates a TruncTensorSeq

julia> matrix_tensorSeq_congruence(a,sig_axis(TTSm))       # a TruncTensorSeqElem
0-dimensional Array{QQMPolyRingElem, 0}:
1
⊕
2-element Vector{QQMPolyRingElem}:
 a₁₁ + a₁₂
 a₂₁ + a₂₂
⊕
2×2 Matrix{QQMPolyRingElem}:
 1//2*a₁₁^2 + a₁₁*a₁₂ + 1//2*a₁₂^2      1//2*a₁₁*a₂₁ + a₁₁*a₂₂ + 1//2*a₁₂*a₂₂
 1//2*a₁₁*a₂₁ + a₂₁*a₁₂ + 1//2*a₁₂*a₂₂  1//2*a₂₁^2 + a₂₁*a₂₂ + 1//2*a₂₂^2
\end{verbatim}
\end{example}

Furthermore, we introduce structs \texttt{FreeTruncSigAlgMultiv} and \texttt{FreeTruncSigAlgMultivElem} that implement the free associative algebra accoding to \eqref{eq:multivariate_free_algeba_Rs}. We implement a constructor for \texttt{FreeTruncSigAlgMultivElem}, called \texttt{free\_sig\_from\_sample},  that returns a sum of variables for the $i$-th sample element in our barycenter computation. 

\begin{example}We illustrate that inversion is polynomial according to \Cref{lem:expLogMulLinIn_ck}. 
\begin{verbatim}
F, s = free_trunc_sig_alg_multiv(3,2)        # FreeTruncSigAlgMultiv for k=3, N=2

julia> inv(free_sig_from_sample(1,F))        # inverse of s₁⁽³⁾ + s₁⁽²⁾+ s₁⁽¹⁾ + 1    
-s₁⁽¹⁾^3 + s₁⁽¹⁾^2 + s₁⁽¹⁾*s₁⁽²⁾ + s₁⁽²⁾*s₁⁽¹⁾ - s₁⁽¹⁾ - s₁⁽²⁾ - s₁⁽³⁾ + 1
\end{verbatim}
\end{example}

In addition to this, we implement a second constructor \texttt{free\_sig\_bary} that returns a sum of variables for the barycenter according to the proof of \Cref{lem:baryIsNonComPolyMap}. With the command \texttt{graded\_component} we project a \texttt{FreeTruncSigAlgMultivElem} to the $j$-th graded component with respect to the level grading.

\begin{example}\label{ex:polynomial_bary}
We compute $f_1,f_2$ and $f_3$ according to \eqref{eq:fj}. 
\begin{verbatim}
F, s = free_trunc_sig_alg_multiv(3,2)        # k = 3, N = 2
s1 = free_sig_from_sample(1,F)               # s₁⁽¹⁾ + s₁⁽²⁾ + s₁⁽³⁾ + 1
s2 = free_sig_from_sample(2,F)               # s₂⁽¹⁾ + s₂⁽²⁾ + s₂⁽³⁾ + 1
y = free_sig_bary(F)                         # y⁽¹⁾ + y⁽²⁾ + y⁽³⁾ + 1
g = QQ(1,2)*sum(log(inv(y)*si) for si in [s1,s2]) 

julia> graded_component(g+y,1)               # f_1
1//2*s₁⁽¹⁾ + 1//2*s₂⁽¹⁾

julia> graded_component(g+y,2)               # f_2
-1//4*s₁⁽¹⁾^2 + 1//4*s₁⁽¹⁾*y⁽¹⁾ - 1//4*s₂⁽¹⁾^2 + 1//4*s₂⁽¹⁾*y⁽¹⁾ - 1//4*y⁽¹⁾*s₁⁽¹⁾ 
- 1//4*y⁽¹⁾*s₂⁽¹⁾ + 1//2*y⁽¹⁾^2 + 1//2*s₁⁽²⁾ + 1//2*s₂⁽²⁾

julia> graded_component(g+y,3)               # f_3
1//6*s₁⁽¹⁾^3 - 1//6*s₁⁽¹⁾^2*y⁽¹⁾ + 1//12*s₁⁽¹⁾*y⁽¹⁾*s₁⁽¹⁾ - 1//12*s₁⁽¹⁾*y⁽¹⁾^2 
+ 1//6*s₂⁽¹⁾^3 - 1//6*s₂⁽¹⁾^2*y⁽¹⁾ + 1//12*s₂⁽¹⁾*y⁽¹⁾*s₂⁽¹⁾ - 1//12*s₂⁽¹⁾*y⁽¹⁾^2 
+ 1//12*y⁽¹⁾*s₁⁽¹⁾^2 - 1//12*y⁽¹⁾*s₁⁽¹⁾*y⁽¹⁾ + 1//12*y⁽¹⁾*s₂⁽¹⁾^2 
- 1//12*y⁽¹⁾*s₂⁽¹⁾*y⁽¹⁾ + 1//6*y⁽¹⁾^2*s₁⁽¹⁾ + 1//6*y⁽¹⁾^2*s₂⁽¹⁾ - 1//3*y⁽¹⁾^3 
- 1//4*s₁⁽¹⁾*s₁⁽²⁾ + 1//4*s₁⁽¹⁾*y⁽²⁾ - 1//4*s₁⁽²⁾*s₁⁽¹⁾ + 1//4*s₁⁽²⁾*y⁽¹⁾ 
- 1//4*s₂⁽¹⁾*s₂⁽²⁾ + 1//4*s₂⁽¹⁾*y⁽²⁾ - 1//4*s₂⁽²⁾*s₂⁽¹⁾ + 1//4*s₂⁽²⁾*y⁽¹⁾ 
- 1//4*y⁽¹⁾*s₁⁽²⁾ - 1//4*y⁽¹⁾*s₂⁽²⁾ + 1//2*y⁽¹⁾*y⁽²⁾ - 1//4*y⁽²⁾*s₁⁽¹⁾ 
- 1//4*y⁽²⁾*s₂⁽¹⁾ + 1//2*y⁽²⁾*y⁽¹⁾ + 1//2*s₁⁽³⁾ + 1//2*s₂⁽³⁾
\end{verbatim}
\end{example}

We implement an evaluation \texttt{eval} from \texttt{TruncTensorSeqElem} to \texttt{FreeTruncSigAlgMultivElem} according to \eqref{eq:evalNonComPolyTensors}. 
With this evaluation we realize the barycenter \texttt{bary} as a non-commutative map constructed according to \Cref{lem:baryIsNonComPolyMap}. Finally,  \texttt{ideal\_of\_entries} interprets all entries in \texttt{FreeTruncSigAlgMultivElem} as relations and returns its ideal. 
In this work, we do not focus on solving the resulting system \eqref{eq:pol_system} for large dimension, number of segments, or truncation levels. For our small systems, we can use the built-in Gr\"obner machinery from \texttt{OSCAR} (in fact \texttt{Singular}). 

\begin{example}\label{ex:lowerBoundB2211_OSCAR}We verify \Cref{ex:lowerBoundB2211} with out \texttt{OSCAR} implementation. 
\begin{verbatim}
R, a = polynomial_ring_sig_transform(2,1)                          # d = 2, m = 1
TTSm = trunc_tensor_seq(R,2,1)                                     # k = 2
sigX1 = matrix_tensorSeq_congruence([QQ(1,2);1],sig_axis(TTSm))
sigX2 = matrix_tensorSeq_congruence([1;QQ(1,2)],sig_axis(TTSm))
sY = bary([sigX1,sigX2])
s_pwln = matrix_tensorSeq_congruence(a,sig_axis(TTSm))             # Ansatz
I = ideal_of_entries(s_pwln - sY)                                  # relations 

julia> groebner_basis(I)                                           # solve 
Gröbner basis with elements
  1: 4*a₂₁ - 3
  2: 4*a₁₁ - 3
with respect to the ordering
  degrevlex([a₁₁, a₂₁])
\end{verbatim}
\end{example}
We show the missing claim from \Cref{fig:Bdkalpha} on path recovery for cubic tensor truncation. For this we use the constructor \texttt{sig\_segment\_standard\_direction} for  \texttt{TruncTensorSeqElem} which returns the signature of a segment $E^j$ according to \Cref{def:canonicalAxisSubpath}. 

\begin{lemma}\label{prop:B2311}
We have $B_{d,3}(1,1)=3$ for every $d\geq 2$. 
\end{lemma}
\begin{proof}
    As shown in the proof of \Cref{prop:well_def_bary_order} we have to perform a single recovery in \texttt{OSCAR}.
\begin{verbatim}
R, a = polynomial_ring_sig_transform(2,3)                       # d=2, m=3
TTSm = trunc_tensor_seq(R,3,3)                                  # k=3
TTSd = trunc_tensor_seq(R,3,2) 
sY = bary([sig_segment_standard_direction(TTSd,j) for j in (1,2)]) 
s_pwln = matrix_tensorSeq_congruence(a,sig_axis(TTSm)) 
I = ideal_of_entries(s_pwln - sY) + ideal(R,[a[1,1]-QQ(3,4)])   # a[1,1]=3/4 

julia> normal_form.(a, Ref(I))
2×3 Matrix{QQMPolyRingElem}:
 3//4  -1//4  0
 1//4  -3//4  1
\end{verbatim}
Note that no path recovery is possible for two segments. This can be shown with an analogous computation but for $m=2$. Then, it is easy to see that the ideal in that case is already the entire ring. 
\end{proof}

\begin{example}\label{ex:B2311}
 In \Cref{fig:Bdkalpha} we illustrate the recovered path from the barycenter of the signatures of $X_1=A\cdot E^1$ and $X_2=A\cdot E^2$ with $A=\begin{bsmallmatrix}1&1\\\frac12&-\frac12\end{bsmallmatrix}$ in the case of truncation level $k=3$. We illustrate the possible recovered path  
 $$Y=A\begin{bsmallmatrix}\frac34&-\frac14&0\\\frac14&-\frac34&1\end{bsmallmatrix}\cdot\axis^3=\begin{bsmallmatrix}1&-1&1\\\frac14&\frac14&-\frac12\end{bsmallmatrix}\cdot\axis^3$$
 using the recovery of $\bary(\sigma(E^1),\sigma(E^2))$ from \Cref{prop:B2311}. In fact, the infinite space of $3$-segment recovered paths is precisely represented by the linear transforms of the axis path,
 \begin{equation}\label{eq:matrixSolutionLastREcovery}\begin{bmatrix}1&-1&1\\\omega&\frac1{8\omega}-\omega&-\frac1{8\omega}\end{bmatrix}\end{equation} with $\omega\not=0$. In \Cref{fig:PathRec_k3_varOmega} we illustrate several recovered paths for different values of $\omega$. 
To show that \eqref{eq:matrixSolutionLastREcovery} precisely represents this set of solutions, we use Gr\"obner bases and see that the system \eqref{eq:pol_system} is the affine variety 
 \begin{equation}\label{eq:varietyLastInfinitRecovery}
 \mathcal{V}(a_{13}-1,
  a_{12}+1,
  a_{21}+a_{22}+a_{23},
  a_{11}+a_{12}+a_{13}-1,
  8a_{22}a_{23} + 8a_{23}^2 - 1).
  \end{equation}
  If $\omega=0$, then the intersection of \eqref{eq:varietyLastInfinitRecovery} with $\mathcal{V}(a_{11}-\omega)$ is empty. Otherwise, the intersection contains the single matrix \eqref{eq:matrixSolutionLastREcovery}. 
 In particular, the second interpolation point of every $3$-segment recovered path is $(0,\frac1{8\omega})$ with $\omega\not=0$.  
Geometrically, this implies that no $3$-segment recovered path $Y$ lies in the convex hull that is spanned by the segments $X_1$ and $X_2$, i.e., $Y([0,1])\nsubseteq\mathsf{conv}\left((0,0),(1,\frac12),(1,-\frac12)\right)$.
\end{example}

\begin{figure}[h]
    \centering
    \begin{subfigure}{0.24\textwidth}
        \centering
\begin{tikzpicture}[scale=0.4]
\begin{axis}[
   xtick={0,1/2,1},
    xticklabels={$0$,$\frac{1}{2}$,$1$}, 
    ytick={-1/2,0,1/4,1/2},
    yticklabels={$-\frac12$,$0$,$\omega$,$\frac12$},
    xmin=-0.2, xmax=1.5,
    ymin=-0.75, ymax=0.95,
    domain=0:1,
]

    % Non-constant part in red (linear interpolation)
    \addplot[CeruleanBlue, very thick] coordinates {(0,0) (1,1/2)};

        \addplot[myGreen, very thick] coordinates {(0,0) (1,-1/2)};

    \addplot[darkred, very thick] coordinates {(0,0) (1,1/4) (0,1/2) (1,0)};

   % Marking the points
    \addplot[mark=*, only marks] coordinates {(0,0) (1,1/4) (0,1/2) (1,0) (1,-1/2) (1,1/2)};

\node[right] at (axis cs:1,0) {$Y$};
\node[right] at (axis cs:1,1/2) {$X_1$};
\node[right] at (axis cs:1,-1/2) {$X_2$};

\end{axis}
\end{tikzpicture}
\caption*{$\omega=\frac14$}
\end{subfigure}
\hfill
        \begin{subfigure}{0.24\textwidth}
        \centering
\begin{tikzpicture}[scale=0.4]
\begin{axis}[
   xtick={0,1/2,1},
    xticklabels={$0$,$\frac{1}{2}$,$1$}, 
    ytick={-1/2,0,353553/1000000,1/2},
    yticklabels={$-\frac12$,$0$,$\omega$,$\frac12$},
    xmin=-0.2, xmax=1.5,
    ymin=-0.75, ymax=0.95,
    domain=0:1,
]

    % Non-constant part in red (linear interpolation)
    \addplot[CeruleanBlue, very thick] coordinates {(0,0) (1,1/2)};

        \addplot[myGreen, very thick] coordinates {(0,0) (1,-1/2)};

    \addplot[darkred, very thick] coordinates {(0,0) (1,353553/1000000) (0,353553/1000000) (1,0)};

   % Marking the points
    \addplot[mark=*, only marks] coordinates {(0,0) (1,353553/1000000) (0,353553/1000000) (1,0) (1,-1/2) (1,1/2)};

\node[right] at (axis cs:1,0) {$Y$};
\node[right] at (axis cs:1,1/2) {$X_1$};
\node[right] at (axis cs:1,-1/2) {$X_2$};

\end{axis}
\end{tikzpicture}
\caption*{$\omega=\frac1{2\sqrt{2}}$}
\end{subfigure}
\hfill
\begin{subfigure}{0.24\textwidth}
        \centering
\begin{tikzpicture}[scale=0.4]
\begin{axis}[
   xtick={0,1/2,1},
    xticklabels={$0$,$\frac{1}{2}$,$1$}, 
    ytick={-1/2,0,1/2,3/4},
    yticklabels={$-\frac12$,$0$,$\frac12$,$\omega$},
    xmin=-0.2, xmax=1.5,
    ymin=-0.75, ymax=0.95,
    domain=0:1,
]

    % Non-constant part in red (linear interpolation)
    \addplot[CeruleanBlue, very thick] coordinates {(0,0) (1,1/2)};

        \addplot[myGreen, very thick] coordinates {(0,0) (1,-1/2)};

    \addplot[darkred, very thick] coordinates {(0,0) (1,3/4) (0,1/6) (1,0)};

   % Marking the points
    \addplot[mark=*, only marks] coordinates {(0,0) (1,3/4) (0,1/6)(1,0) (1,-1/2) (1,1/2)};

\node[right] at (axis cs:1,0) {$Y$};
\node[right] at (axis cs:1,1/2) {$X_1$};
\node[right] at (axis cs:1,-1/2) {$X_2$};

\end{axis}
\end{tikzpicture}
\caption*{$\omega=\frac34$}
\end{subfigure}
\hfill
\begin{subfigure}{0.24\textwidth}
        \centering
\begin{tikzpicture}[scale=0.4]
\begin{axis}[
   xtick={0,1/2,1},
    xticklabels={$0$,$\frac{1}{2}$,$1$}, 
    ytick={-1/2,-1/4,0,1/2},
    yticklabels={$-\frac12$,$\omega$,$0$,$\frac12$},
    xmin=-0.2, xmax=1.5,
    ymin=-0.75, ymax=0.95,
    domain=0:1,
]

    % Non-constant part in red (linear interpolation)
    \addplot[CeruleanBlue, very thick] coordinates {(0,0) (1,1/2)};

        \addplot[myGreen, very thick] coordinates {(0,0) (1,-1/2)};

    \addplot[darkred, very thick] coordinates {(0,0) (1,-1/4) (0,-1/2) (1,0)};

   % Marking the points
    \addplot[mark=*, only marks] coordinates {(0,0) (1,-1/4) (0,-1/2) (1,0) (1,-1/2) (1,1/2)};

\node[right] at (axis cs:1,0) {$Y$};
\node[right] at (axis cs:1,1/2) {$X_1$};
\node[right] at (axis cs:1,-1/2) {$X_2$};

\end{axis}
\end{tikzpicture}
\caption*{$\omega=-\frac14$}
\end{subfigure}
\caption{Recovered paths $Y$ parametrized by $\omega\in\R\setminus\{0\}$ for $\bary\left(\sigma(X_1),\sigma(X_2)\right)$ in truncation level $k=3$ according to \Cref{ex:B2311,ex:computeAreaOfBarycenterRecovery}.}
\label{fig:PathRec_k3_varOmega}
\end{figure}
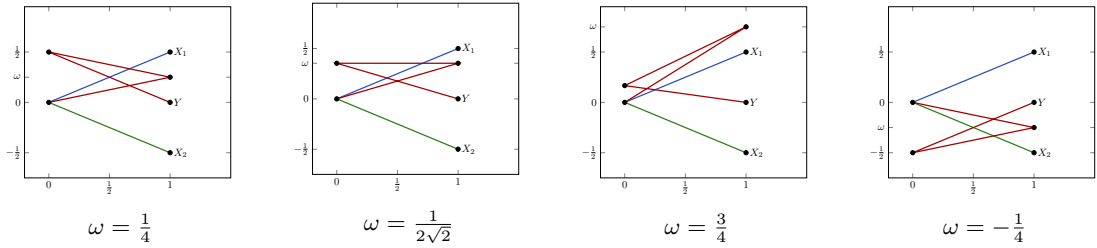

\begin{example}\label{ex:computeAreaOfBarycenterRecovery}
We use geometric methods to compute the signed area of the path represented by \eqref{eq:matrixSolutionLastREcovery}, i.e., the polygonal loop given by the columns of 
\begin{equation}\label{eq:polygonWithAreaZero}
\begin{bmatrix}
0&1&0&1&0\\
0&\omega&\frac1{8\omega}&0&0
\end{bmatrix}
\end{equation}
in ascending order and $\omega\in\R\setminus\{0\}$.
See \Cref{fig:PathRec_k3_varOmega} for an illustration with various choices of $\omega$. The polygonal loop \eqref{eq:polygonWithAreaZero} is self-intersecting with one intersection point 
$\left(x_\omega,y_\omega\right):=(\frac1{8\omega^2+1},\frac\omega{8\omega^2+1})$. 
Therefore we have to compute the areas of the two triangle regions 
\begin{equation}\label{eq:triangels}\begin{bmatrix}
0&x_\omega&1&0\\
0&y_\omega&0&0
\end{bmatrix}\text{ and }\begin{bmatrix}
x_\omega&1&0&x_\omega\\
y_\omega&\omega&\frac1{8\omega} &y_\omega  
\end{bmatrix},
\end{equation}
which coincide using  \emph{Gau{\ss}'s area formula} \cite[p. 125]{koecher2007ebene},
\begin{align*}
&\frac12\left\lvert
\det\left(\begin{matrix}0&x_\omega\\0 &y_\omega\end{matrix}\right)+
\det\left(\begin{matrix}x_\omega&1\\y_\omega &0\end{matrix}\right)
+
\det\left(\begin{matrix}1&0\\0 &0\end{matrix}\right)
\right\rvert\\
&=
\frac12\left\lvert\det\left(\begin{matrix}x_\omega&1\\y_\omega &\omega\end{matrix}\right)
+
\det\left(\begin{matrix}1&0\\\omega&\frac1{8\omega} \end{matrix}\right)
+
\det\left(\begin{matrix}0&x_\omega\\\frac1{8\omega}&y_\omega \end{matrix}\right)\right\rvert=\frac12y_\omega
\end{align*}
also known as the  \emph{shoelace formula}; see \cite[Chapter 13.1.3]{laaksonen2020guide}.  The \emph{winding numbers} of the triangles \eqref{eq:triangels} cancel each other out, so the polygonal loop \eqref{eq:polygonWithAreaZero} has zero signed area. 
This illustrates \Cref{rem:polynomialViaBCH} with the factorization according to \eqref{eq:aBCHk2}, showing that the signed area of the barycenter is the mean of the signed areas of the sample. Segments have zero signed area (see \cite[Theorem 1.4]{FLS24} or \cite[Corollary 7.3]{AGOSS23}), so we already know that the recovered path represented by \eqref{eq:matrixSolutionLastREcovery} has zero signed area. 
\end{example}

\section*{Data availability}
The implementation of our proposed methods, along with the data and experiments, is publicly available at \url{https://github.com/leonardSchmitz/barycenter-signature}.

\section*{Acknowledgements}

We thank Antony Della Vecchia and Marcel Wack for several suggestions and comments on earlier versions of our \texttt{OSCAR} implementation. 

The authors acknowledge support from DFG CRC/TRR 388 ``Rough Analysis, Stochastic Dynamics and Related Fields'', Project A04.

\bibliographystyle{siam} % or whatever style you used
\bibliography{refs.bib}

\end{document}